\documentclass[12pt]{article}
\usepackage[pdftex,bookmarksopen=true,bookmarks=true,unicode,setpagesize]{hyperref}
\hypersetup{colorlinks=true,linkcolor=black,citecolor=black}

\usepackage{amssymb, amsmath, amsthm}

\numberwithin{equation}{section}

\allowdisplaybreaks
\usepackage{cite}

\newcommand\R{\mathbb{R}}

\renewcommand\i{{\rm 1\kern -.3600em 1}}

\newtheorem{theorem}{Theorem}[section]
\newtheorem{corollary}[theorem]{Corollary}
\newtheorem{lemma}[theorem]{Lemma}
\newtheorem{proposition}[theorem]{Proposition}

\theoremstyle{remark}

\theoremstyle{remark}

\theoremstyle{remark}
\newtheorem{remark}[theorem]{Remark}

\newcommand{\so}{\mathbf s}
\newcommand{\co}{\mathbf c}
\newcommand{\SO}{\mathbf S}
\newcommand{\LO}{\mathbf L}

\begin{document}

\vspace{-20mm}
\begin{center}{\Large \bf
Stirling operators in spatial combinatorics}
\end{center}

{\large Dmitri Finkelshtein}\\ Department of Mathematics, Computational Foundry, Swansea University, Bay Campus,  Swansea SA1 8EN, U.K.;
e-mail: \texttt{d.l.finkelshtein@swansea.ac.uk}\vspace{2mm}

{\large Yuri Kondratiev}\\ Fakult\"at f\"ur Mathematik, Universit\"at Bielefeld,
            33615~Bielefeld, Germany;\\
e-mail: \texttt{kondrat@mathematik.uni-bielefeld.de}\vspace{2mm}

{\large Eugene Lytvynov}\\ Department of Mathematics, Computational Foundry, Swansea University, Bay Campus,  Swansea SA1 8EN, U.K.;
e-mail: \texttt{e.lytvynov@swansea.ac.uk}\vspace{2mm}

{\large Maria Jo\~{a}o Oliveira}\\ DCeT, Universidade Aberta,
            1269-001 Lisbon, Portugal; CMAFCIO, University of Lisbon, 1749-016 Lisbon, Portugal;\\
e-mail: \texttt{mjoliveira@ciencias.ulisboa.pt}\vspace{2mm}


{\small
\begin{center}
{\bf Abstract}
\end{center}
\noindent We define and study a spatial (infinite-dimensional) counterpart of Stirling numbers. In classical combinatorics, the Pochhammer symbol $(m)_n$ can be extended from a natural number $m\in\mathbb N$ to the falling factorials $(z)_n=z(z-1)\dotsm (z-n+1)$ of an argument $z$ from $\mathbb F=\mathbb R\text{ or }\mathbb C$, and Stirling numbers of the first and second kinds are the coefficients of the expansions of $(z)_n$ through $z^k$, $k\leq n$ and vice versa. When taking into account spatial positions of elements in a locally compact Polish space $X$, we replace $\mathbb N$ by the space of configurations---discrete Radon measures $\gamma=\sum_i\delta_{x_i}$ on $X$, where $\delta_{x_i}$ is the Dirac measure with mass at $x_i$. The spatial falling factorials $(\gamma)_n:=\sum_{i_1}\sum_{i_2\ne i_1}\dotsm\sum_{i_n\ne i_1,\dots, i_n\ne i_{n-1}}\delta_{(x_{i_1},x_{i_2},\dots,x_{i_n})}$ can be naturally extended to  mappings $M^{(1)}(X)\ni\omega\mapsto (\omega)_n\in M^{(n)}(X)$, where $M^{(n)}(X)$ denotes the space of $\mathbb F$-valued, symmetric (for $n\ge2$) Radon measures on $X^n$. There is a natural duality between $M^{(n)}(X)$ and the space $\mathcal {CF}^{(n)}(X)$ of $\mathbb F$-valued, symmetric  continuous functions on $X^n$ with compact support.  The Stirling operators of the first and second kind, $\mathbf{s}(n,k)$ and $\mathbf{S}(n,k)$, are  linear operators, acting between spaces $\mathcal {CF}^{(n)}(X)$ and $\mathcal {CF}^{(k)}(X)$ such that their dual operators, acting from $M^{(k)}(X)$ into $M^{(n)}(X)$, satisfy  $(\omega)_n=\sum_{k=1}^n\mathbf{s}(n,k)^*\omega^{\otimes k}$ and $\omega^{\otimes n}=\sum_{k=1}^n\mathbf{S}(n,k)^*(\omega)_k$, respectively.  In the case where $X$ has only a single point, the Stirling operators can be identified with Stirling numbers.  We derive combinatorial properties of the Stirling operators, present their connections with a generalization of the Poisson point process and with the Wick ordering under the canonical commutation relations.
\vspace{2mm}

\vspace{2mm}


\section{Introduction}\label{rts54q43}

The classical combinatorics deals with finite structures, e.g.\ cardinalities of
finite sets, which, in particular, may represent populations. This brings clear
intuitive interpretations of numerous combinatorial objects, especially in their
probabilistic applications, see e.g.\ {\cite[Chapter~II]{Feller}}. The state
space is hence the set of nonnegative integers, $\mathbb N_0=\{0,1,2,\dots\}$. 
Defined initially for numbers from $\mathbb N_0$, many combinatorial objects admit natural extensions to  a variable from $\mathbb F=\mathbb R$ or $\mathbb C$. This leads to important relations between combinatorics and analysis, in particular,   difference calculus.
 
 When studying population dynamics in biology or ecology, one is often interested not only in the size of the population in a certain region but also how the population is spatially distributed, see e.g.\  \cite{DLM}. Let $X$ denote the space in which the population is located. For a sufficient generality, we assume that $X$ is a locally compact Polish space. The spatial distribution of a population in $X$ is modelled by a configuration $\gamma$, a locally finite subset of $X$. We denote by $\Gamma(X)$ the space of all configurations $\gamma$ in $X$.

We note that configuration spaces are widely  used in different branches of mathematics and applications. We only mention the theory of point processes (e.g.\ \cite{DVJ}), statistical physics of continuous systems (e.g.\ \cite{Preston}), geometry and topology (e.g.\ \cite{AKR,FH}), spatial ecology (e.g.\ \cite{DLM}).
 Contrary  to $\mathbb N_0$, the configuration space $\Gamma(X)$ possesses both continuous topological properties (arising from the topology in $X$)
  and  the discrete structure of each particular configuration.\footnote{Although a configuration is a locally finite set, the total number of points of a configuration can be infinite. This feature is important for many branches of mathematics, in particular, for the theory of point processes and statistical physics of continuous systems.} 
  
 Having in mind the spatial distribution of a population, we may naturally think about extending certain notions and results of the classical combinatorics to the configuration space $\Gamma(X)$. The aim of this paper is to show that  Stirling numbers and many related results have their natural counterparts in the  spatial combinatorics. In turn, our results within the spatial combinatorics cast new light on the classical combinatorics related to Stirling numbers.

Let us now briefly describe what we mean under a spatial counterpart of  Stirling numbers. Let $n\in\mathbb N=\{1,2,3,\dots\}$. For $k\in\mathbb N_0$, $(k)_n=\frac1{n!}\binom kn= k(k-1)\dotsm(k-n+1)$ is called a {\it falling factorial}.  The latter expression allows one to define  falling factorials as polynomials of a variable $z\in\mathbb F$: $(z)_n=z(z-1)\dotsm(z-n+1)$. The {\it Stirling numbers of the first kind}, $s(n,k)$, are defined through the expansion $(z)_n=\sum_{k=1}^ns(n,k)z^k$, while the {\it Stirling numbers of the second kind}, $S(n,k)$, are defined through the inverse expansion $z^n=\sum_{k=1}^nS(n,k)(z)_k$. We refer e.g.\ to the monographs \cite{QG,MS} and the references therein for studies of  Stirling numbers in combinatorics and their applications in mathematical physics.
 
 It is standard to interpret a configuration $\gamma=\{x_i\}\in\Gamma(X)$ as a discrete  measure on $X$: $\gamma=\sum_{i}\delta_{x_i}$ where $\delta_{x_i}$ is the Dirac measure with mass at $x_i$. Then, for $n\in\mathbb N$, one can naturally define a {\it (spatial) falling factorial} $(\gamma)_n$ as the discrete measure on $X^n$ given by
 \begin{equation}\label{rts5w53w}
 (\gamma)_n:=\sum_{x_1\in\gamma}\sum_{x_2\in\gamma\setminus\{x_1\}}\dotsm\sum_{x_n\in\gamma\setminus\{x_1,\dots,x_{n-1}\}}\delta_{x_1}\otimes\delta_{x_2}\otimes\dots\otimes\delta_{x_n},\end{equation}
 see \cite{FKLO}. 
 If $A\subset X$ is compact and $\gamma(A)=k\in\mathbb N_0$ (i.e., the configuration $\gamma$ has $k$ points in $A$), then $(\gamma)_n(A^n)=(k)_n$\,.
Similarly to the classical case, one can now extend the definition of the falling factorials $(\gamma)_n$ to a linear space.  More precisely, let $M(X)$ denote the space of $\mathbb F$-valued Radon measures on $X$. For each $\omega\in M(X)$,  we define an $\mathbb F$-valued Radon measure $(\omega)_n$ on $X^n$ by\footnote{In the case where $X$ is a smooth manifold, the definition of the falling factorials can be further extended to the case where $\omega$ is a generalized function on $X$,  see \cite{BKKL,FKLO,KunaInfusino}. This brings further connections with infinite dimensional analysis. However, in this paper, we will not discuss this wider extension.} 
\begin{align}
&(\omega)_n(dx_1\dotsm dx_n):=\omega(dx_1)\big(\omega(dx_2)-\delta_{x_1}(dx_2)\big)\notag\\
&\qquad\times\dotsm\times\big(\omega(dx_n)-\delta_{x_1}(dx_n)-\delta_{x_2}(dx_n)-\dots-\delta_{x_{n-1}}(dx_n)\big).\label{xesa332a}\end{align}
Each $(\omega)_n$ is a symmetric measure, i.e., it remains invariant under the natural action of the symmetric group on $X^n$.

 For $n\in\mathbb N$, let $\mathcal F^{(n)}(X)$ denote the space of measurable, bounded, compactly supported, symmetric functions $f^{(n)}:X^n\to\mathbb F$. We define {\it Stirling operators of the first kind} as the linear operators 
 $\so(n,k):\mathcal F^{(n)}(X)\to\mathcal F^{(k)}(X)$ satisfying
\begin{equation}\label{cxta5wq5y43w}
\int_{X^n}f^{(n)}\,d(\omega)_n=\sum_{k=1}^n\int_{X^k}\so(n,k)f^{(n)}\,d\omega^{\otimes k},\quad\text{for all }\omega\in M(X),\end{equation}
 and {\it Stirling operators of the second kind} as the linear operators 
 $\SO(n,k):\mathcal F^{(n)}(X)\to\mathcal F^{(k)}(X)$ satisfying
\begin{equation}\label{cde5w35}
\int_{X^n}f^{(n)}\,d\omega^{\otimes n}=\sum_{k=1}^n\int_{X^k}\SO(n,k)f^{(n)}\,d(\omega)_k, \quad\text{for all }\omega\in M(X).\end{equation}
In the case where the underlying space $X$ has a single point (hence the whole population is located at this point), we may obviously identify the operators $\so(n,k)$ and $\SO(n,k)$ with the Stirling numbers $s(n,k)$ and $S(n,k)$, respectively. In the case where $X=\{x_1,\dots,x_n\}$ with $n\in\mathbb N$, $n\ge2$,  the Stirling operators $\so(n,k)$ and $\SO(n,k)$ act in finite-dimensional spaces, hence they are multivariate extensions of the corresponding Stirling numbers.

Let us now describe the content of the paper. In Section \ref{sa45wq43q}, we discuss the preliminaries. In particular, following \cite{FKLO}, we discuss the binomial property of the falling factorials $(\omega)_n$ and their lowering operators. 

In Section~\ref{sa4t2q}, we discuss basic results on the Stirling operators. These include the explicit formulas for the action of the Stirling operators, the recurrence relations satisfied by the Stirling operators, and the explicit form of their generating functions. We also introduce  {\it Lah operators}, which connect the rising factorials, $(\omega)^{(n)}=(-1)^n(-\omega)_n$\,, with the falling factorials, $(\omega)_n$. We derive the explicit form of the Lah operators and their generating function. 

In Section \ref{few53q}, we discuss an infinite dimensional counterpart of Euler's formula for the Stirling operators of the second kind,  compare with \cite[Section 9.1]{QG}. The proof of our result uses the binomial property of the falling factorials, $(\omega)_n$.  We also discuss how the infinite dimensional Euler formula is related to the correlation measure of a point process in $X$. 

Olson's identity (see e.g.\ \cite[Section 12.2]{QG}) is a generalization of the orthogonality identity satisfied by the Stirling numbers. In Section~\ref{tes5w}, by using  Euler's formula from Section \ref{few53q}, we prove an infinite dimensional counterpart of Olson's identity.  Even in the case of a single-point space $X$, the obtained identity yields an extension of the classical Olson's identity for  Stirling numbers. We also discuss a couple of other identities satisfied by the Stirling operators. The proofs of these identities significantly use the binomial property of  $(\omega)_n$. 

Our studies of the Stirling operators naturally lead us to a generalization of  Poisson point process. Recall that, if $\sigma$ is a non-atomic positive Radon measure on $X$, one can define the Poisson point process, with intensity measure $\sigma$ as a probability measure on the configuration space $\Gamma(X)$. We denote by $\mathbb E_\sigma$ the expectation with respect to this probability measure. Let $p$ be a polynomial on $M(X)$, i.e., a function $p:M(X)\to\mathbb F$ which is of the form 
\begin{equation}\label{dftwfedt}
p(\omega)=f^{(0)}+\sum_{k=1}^n\int_{X^k}f^{(k)}\,d\omega^{\otimes k},\end{equation}
where $f^{(0)}\in\mathbb F$ and $f^{(k)}\in\mathcal F^{(k)}(X)$ for $k=1,\dots,n$. 
Then $p$ is integrable  and we can evaluate its expectation, $\mathbb E_\sigma(p)$. If we denote by $\mathcal P(M(X))$ the space of all polynomials on $M(X)$,  the expectation 
$\mathbb E_\sigma$ determines a linear functional $\mathbb E_\sigma:\mathcal P(M(X))\to\mathbb F$. In Appendix, we prove that such a functional can be naturally defined for any $\mathbb F$-valued Radon measure $\sigma\in M(X)$. We call $\mathbb E_\sigma$ the {\it Poisson functional with intensity measure} $\sigma$. We prove that several standard properties of the Poisson point process (for example, Mecke identity, see Proposition~\ref{g6ew4}) still hold for the Poisson functional $\mathbb E_\sigma$. The main result of Section~\ref{ufxdqyr} (Theorem~\ref{e6ew63}) states a connection between the Poisson functional $\mathbb E_\sigma$ and the Stirling operators of the second kind. More precisely, for a monomial $p(\omega)=\int_{X^n}f^{(n)}\,d\omega^{\otimes n}$, it holds that 
$\mathbb E_\sigma(p)=\sum_{k=1}^n\int_{X^k}\SO(n,k)f^{(n)}\,d\sigma^{\otimes k}$.

Let $a^+$, $a^-$ be a pair of (adjoint) operators satisfying the canonical commutation relation 
\begin{equation}\label{xzzaaaa}
a^-a^+=a^+a^-+1,\end{equation}
where  $a^+$ is called a {\it creation operator} and $a^-$ an {\it annihilation operator}. The operator $\rho=a^+a^-$ is called a {\it number operator}, or a {\it particle density}. Due to the commutation relation, the operator $\rho^n$ can be represented as a linear combination of the {\it Wick (normally) ordered operators}, $(a^+)^k(a^-)^k$. Katriel \cite{Katriel} in 1974 found the explicit formula, $\rho^n=\sum_{k=1}^nS(n,k)(a^+)^k(a^-)^k$. Katriel's result includes, as a special case,  the classical formula of Gr\"unert (1843) for the action of the Euler operator $t\frac d{dt}$,
$$\bigg(t\frac{d}{dt}\bigg)^n=\sum_{k=1}^nS(n,k)\,t^k\bigg(\frac{d}{dt}\bigg)^k.$$ 
 The seminal paper \cite{Katriel} led to numerous generalizations of  Stirling numbers; we refer to \cite[Subsection~ 1.2.3]{MS} for a long list of references. The main result of Section~\ref{gwfdytwyc} (Theorem~\ref{taw4aq234q}) is an extension of Katriel's formula to the case of a family of operators $(a^+(x),a^-(x))_{x\in X}$ satisfying the canonical commutation relations, see formula~\eqref{r4wqa4y} below. In fact, it is only in the case where $X$ is a discrete set that the operators $a^+(x)$ and $a^-(x)$ are well-defined. In the general case, these are operator-valued distributions, hence one needs to choose a reference measure $\sigma$ on $X$ and consider smeared operators $a^+(\varphi) =\int_X\varphi(x)a^+(x)\sigma(dx)$, $a^-(\varphi)=\int_X\varphi(x)a^-(x)\sigma(dx)$, and $\rho(\varphi)=\int_X\varphi(x)a^+(x)a^-(x)\sigma(dx)$. 
Theorem~\ref{taw4aq234q} states the explicit formula for the product $\rho(\varphi_1)\dotsm\rho(\varphi_n)$ in terms of the Wick ordered terms. This formula uses the Stirling operators $\SO(n,k)$ at the place of the Stirling numbers $S(n,k)$ in Katriel's formula. We also consider an infinite dimensional generalization of Gr\"unert's formula (Theorem~\ref{dwq3q32q}), which appears to be a special case of the formula from Theorem~\ref{taw4aq234q}. 

In Theorem~\ref{y6weu43}, for an arbitrary measure $\sigma\in M(X)$,  we consider a certain  representation of the canonical commutation relations and a {\it vacuum functional} $\tau_\sigma$ on the commutative unital algebra generated by the particle densities  $\rho(\varphi)$ and the identity operator. We show that the functional $\tau_\sigma$ can be naturally identified with the Poisson functional with intensity measure $\sigma$. This extends the available results on the Fock space realization of the classical Poisson point process (e.g.\ \cite{GGPS,HP,Surgailis}) to the case of an $\mathbb F$-valued intensity measure $\sigma$.

The classical {\it Touchard (or exponential) polynomials} are given by 
$$T_n(z)=\sum_{k=1}^nS(n,k)z^k,$$
 while $B_n=T_n(1)$ are called the {\it Bell numbers}. ($B_n$ counts the number of all unordered partitions of a set of $n$ elements.)
In Section~\ref{ydrsdese3aqw}, for each $n\in\mathbb N$ and $\omega\in M(X)$,  we define an $\mathbb F$-valued Radon measure $T_n(\omega)$ on $X^n$ that satisfies 
$$\int_{X^n}f^{(n)}\,dT_n(\omega)=\sum_{k=1}^n\int_{X^k}\SO(n,k)f^{(n)}\,d\omega^{\otimes k}\quad\text{for all }f^{(n)}\in \mathcal F^{(n)}(X).$$
We call $T_n$ {\it Touchard (or exponential)  polynomials on} $M(X)$.  Comparing the definition of $T_n(\omega)$ with Theorem~\ref{e6ew63} gives us an immediate connection between $T_n(\omega)$ and the Poisson functional~$\mathbb E_\omega$ (Corollary~\ref{sesawawqa2q}). We prove several further properties of $T_n(\omega)$, including their explicit form, connections with the infinite dimensional Gr\"unert's formula, the binomial property, recurrence formulas, and the explicit form of their generating function.
If $\nu$ is a probability measure on $X$, we call $B_n(\nu)=T_n(\nu)$ the {\it $n$th Bell measure corresponding to~$\nu$}. Proposition~\ref{qwdwswxe} shows that $B_n(\nu)$ is related to the expectation of a certain infinite sum involving an infinite sequence of independent random variables that have distribution~$\nu$.

Finally, in Section \ref{xzrwaq4yq4q}, we discuss an open problem related to Theorem~\ref{taw4aq234q}.

\section{Preliminaries}\label{sa45wq43q}

\subsection{Polynomials on Radon measures}

Let $X$ be a locally compact Polish space. Let $\mathcal B(X)$ denote the Borel $\sigma$-algebra on $X$ and let $\mathcal B_0(X)$ denote the collection of all precompact sets from $\mathcal B(X)$. Recall that a positive Radon measure $\sigma$ on  $(X,\mathcal B(X))$ is  a  measure satisfying $\sigma(A)<\infty$ for all $A\in\mathcal B_0(X)$; a real-valued (signed) Radon measure $\omega$ on $X$ has the form $\omega=\sigma_1-\sigma_2$, where $\sigma_1$ and $\sigma_2$ are positive Radon measures; and a complex-valued Radon  measure $\omega$ on $X$ has the form $\omega=\omega_1+i\omega_2$, where $\omega_1$ and $\omega_2$ are real-valued  Radon measures. Note that, generally speaking,  a real-valued or complex-valued Radon measure  is well-defined on $\mathcal B_0(X)$ only.
Let $\mathbb F$ denote either $\R$ or $\mathbb C$, and we denote by  $M(X)$ the set of all $\mathbb F$-valued Radon measures on $X$. 

A configuration $\gamma$ in $X$ is a subset of $X$ that contains only a finite number of elements in each compact set $K\subset X$. In particular, the set $\gamma$ is either  finite (possibly empty) or countable. We denote by $\Gamma(X)$ the collection of all configurations in $X$.  By identifying each configuration $\gamma=\{x_i\}$ with the (positive) Radon measure $\gamma=\sum_i\delta_{x_i}$, we get the inclusion
$\Gamma(X)\subset M(X)$.

We denote by $M^{(n)}(X)$ the 
set of all symmetric Radon measures on $X^n$.  Obviously, for each $\omega\in M(X)$, the product measure $\omega^{\otimes n}$ belongs to $M^{(n)}(X)$. Let also $M^{(1)}(X):=M(X)$, $M^{(0)}(X):=\mathbb F$ and $\omega^{\otimes 0}:=1$ for $\omega\in M(X)$.

We denote by $\mathcal F(X)$ the space of all $\mathbb F$-valued bounded measurable functions on $X$ with compact support, and by $\mathcal{CF}(X)$ the space of all $\mathbb F$-valued continuous functions on $X$ with compact support.
For each $n\ge 2$, we similarly define the spaces $\mathcal F(X^n)$ and $\mathcal {CF}(X^n)$. We denote by $\mathcal F^{(n)}(X)$ and $\mathcal{CF}^{(n)}(X)$ the spaces of all symmetric functions from $\mathcal F(X^n)$ and $\mathcal {CF}(X^n)$, respectively. We also denote $\mathcal F^{(1)}(X):=\mathcal F(X)$, $\mathcal{CF}^{(1)}(X):=\mathcal{CF}(X)$ and $\mathcal F^{(0)}(X):=\mathbb F$.

For $\mu^{(n)}\in M^{(n)}(X)$ and $f^{(n)}\in\mathcal F^{(n)}(X)$, we denote
\begin{equation}\label{dera4tq2y4}
\langle\mu^{(n)},f^{(n)}\rangle:=\int_{X^n}f^{(n)}\,d\mu^{(n)},\quad n\in\mathbb N,\end{equation}
and $\langle\mu^{(0)},f^{(0)}\rangle:=\mu^{(0)}f^{(0)}$ for $n=0$.  


Let $\mathfrak S_n$ denote the symmetric group of degree $n$.
For a function $f^{(n)}:X^n\to\mathbb F$, we denote by $P_nf^{(n)}$ the  symmetrization of $f^{(n)}$,
$$(P_nf^{(n)})(x_1,\dots,x_n):=\frac1{n!}\sum_{\pi\in \mathfrak S_n}f(x_{\pi(1)},\dots,x_{\pi(n)}).$$
 (For $n=1$, $P_1$ is just the identity.)

For functions $f^{(n)}:X^n\to\mathbb F$ and $g^{(m)}:X^m\to\mathbb F$, we define $f^{(n)}\otimes g^{(m)}:X^{n+m}\to\mathbb F$ by $(f^{(n)}\otimes g^{(m)})(x_1,\dots,x_{n+m}):=f^{(n)}(x_1,\dots,x_n)g^{(m)}(x_{n+1},\dots,x_{n+m})$. 
 For any $f^{(n)}\in\mathcal F^{(n)}(X)$ and $g^{(m)}\in\mathcal F^{(m)}(X)$, the symmetric tensor product of $f^{(n)}$ and $g^{(m)}$ is the function 
$$f^{(n)}\odot g^{(m)}:=P_{n+m}(f^{(n)}\otimes g^{(m)})\in\mathcal F^{(n+m)}(X).$$
 Obviously, if $f^{(n)}\in\mathcal {CF}^{(n)}(X)$ and $g^{(m)}\in\mathcal {CF}^{(m)}(X)$, then $f^{(n)}\odot g^{(m)}\in\mathcal {CF}^{(n+m)}(X)$. 
 
 Similarly, for $\mu^{(n)}\in M^{(n)}(X)$ and $\nu^{(m)}\in M^{(m)}(X)$, we define the symmetric product measure $\mu^{(n)}\odot\nu^{(m)}\in M^{(n+m)}(X)$. This measure can be characterized as the unique element of $M^{(n+m)}(X)$ that satisfies
$$\langle \mu^{(n)}\odot\nu^{(m)} ,f^{(n+m)}\rangle=\langle \mu^{(n)}\otimes\nu^{(m)} ,f^{(n+m)}\rangle\quad \text{for all }f^{(n+m)}\in\mathcal F^{(n+m)}(X). $$

By the polarization identity, each $\mu^{(n)}\in M^{(n)}(X)$ is uniquely characterized by the values $\langle \mu^{(n)},\xi^{\otimes n}\rangle$ with $\xi\in \mathcal {F}(X)$.  Note that, for $\mu^{(0)}\in M^{(0)}(X)=\mathbb F$ and $\xi\in\mathcal F(X)$,  $\langle \mu^{(0)},\xi^{\otimes 0}\rangle=\mu^{(0)}$.

A {\it polynomial on $M(X)$} is a function $p:M(X)\to\mathbb F$ of form \eqref{dftwfedt}. Using the notation~\eqref{dera4tq2y4}, we may write formula \eqref{dftwfedt} as follows:
\begin{equation}\label{tew53w5}
p(\omega)=\sum_{k=0}^n\langle \omega^{\otimes k},f^{(k)}\rangle.\end{equation}
If $f^{(n)}\ne 0$, we say that $p$ is a polynomial of degree $n$. Recall that $\mathcal P(M(X))$ denotes the space of all polynomials on $M(X)$.

 The space $M(X)$ can be naturally equipped with the vague topology.\footnote{The vague topology on $M(X)$ is the minimal topology on $M(X)$ such that, for each $f\in\mathcal {CF}(X)$, the  mapping $M(X)\ni\omega\mapsto\langle\omega,f\rangle\in\mathbb F$ is continuous.} Thus, if in formula \eqref{tew53w5}, $f^{(k)}$  belongs to $\mathcal {CF}^{(k)}(X)$ for each $k=1,\dots,n$, then $p$ is a continuous function on $M(X)$. This is why the subset of $\mathcal P(M(X))$ consisting of all polynomials of form \eqref{tew53w5} with $f^{(k)}\in \mathcal {CF}^{(k)}(X)$ will be called the {\it space of continuous polynomials on $M(X)$} and denoted by $\mathcal {CP}(M(X))$.

\subsection{Falling and rising factorials on Radon measures}

Following \cite{FKLO} (in particular, Subsections 5.1 and 5.2), we will now 
recall some basic properties of the falling and rising factorials on $M(X)$.
 For each $n\in\mathbb N$, the {\it falling factorial on $M(X)$ of degree $n$} is defined to be the mapping
$M(X)\ni\omega\mapsto (\omega)_n\in M^{(n)}(X)$
given by \eqref{xesa332a}. (It is not difficult to see that the  Radon measure on the right hand side of formula \eqref{xesa332a} is indeed symmetric.) Similarly, for each $n\in\mathbb N$, the {\it rising factorial on $M(X)$ of degree $n$} is defined to be the mapping
$M(X)\ni\omega\mapsto (\omega)^{(n)}\in M^{(n)}(X)$
given by
\begin{align}
&(\omega)^{(n)}(dx_1\dotsm dx_n):=\omega(dx_1)\big(\omega(dx_2)+\delta_{x_1}(dx_2)\big)\notag\\
&\qquad\times\dotsm\times\big(\omega(dx_n)+\delta_{x_1}(dx_n)+\delta_{x_2}(dx_n)+\dots+\delta_{x_{n-1}}(dx_n)\big).\label{fdeseas}\end{align}
Clearly, 
\begin{equation}\label{ctsrea4w}
(\omega)^{(n)}=(-1)^n(-\omega)_n,\quad n\in\mathbb N.\end{equation}
We will also  use the notation $(\omega)_0=(\omega)^{(0)}:=1$.

It follows directly from formulas \eqref{xesa332a} and \eqref{fdeseas} that, for each  $\omega\in M(X)$ and any $A\in\mathcal B_0(X)$, 
\begin{equation}\label{s5w53w}
(\omega)_n(A^n)=(\omega(A))_n,\quad (\omega)^{(n)}(A^n)=(\omega(A))^{(n)}.\end{equation}
Here, for $z\in\mathbb F$, $(z)_n$ and $(z)^{(n)}=(-1)^n(-z)_n$ are the classical falling and rising factorials, respectively. 

\begin{remark}\label{fdryde65re4w}
Assume that the underlying space $X$ has only a single point. In that case, we can identify each $M^{(n)}(X)$ with $\mathbb F$ and formula \eqref{s5w53w} means that $(\omega)^{(n)}$ and $(\omega)_n$ are the classical rising and falling factorials, respectively.
\end{remark}

For $\omega\in M(X)$ and $n\in\mathbb N$, we define the {\it (spatial) binomial coefficient}
$\binom \omega n:=\frac1{n!}\,(\omega)_n$.
For each configuration $\gamma=\sum_{i\ge1}\delta_{x_i}\in\Gamma(X)$,
\begin{equation}\label{reaq43q42}
\binom \gamma n=\sum_{\{i_1,\dots,i_n\}} \delta_{x_{i_1}}\odot \delta_{x_{i_2}}\odot\dots\odot \delta_{x_{i_n}},\end{equation}
compare with \eqref{rts5w53w}.

\begin{remark}\label{c4q4q}
It follows from \eqref{rts5w53w} or \eqref{reaq43q42} that, if a configuration $\gamma$ has strictly less than $n$ points, then $\binom \gamma n=(\gamma)_n=0$.
\end{remark}

Both the falling factorials and the rising factorials have the binomial property:
\begin{align}\label{5w53w2}
(\omega+\sigma)_n&=\sum_{k=0}^n\binom nk(\omega)_k\odot(\sigma)_{n-k},\\
\label{5w53w22} (\omega+\sigma)^{(n)}&=\sum_{k=0}^n\binom nk(\omega)^{(k)}\odot(\sigma)^{(n-k)}, \quad\omega,\sigma\in M(X)
\end{align}
 Furthermore, the following lowering property holds: for each $n\in\mathbb N_0$, $\omega\in M(X)$, and $x\in X$,
\begin{equation}\label{s5wuwu}
(\omega+\delta_x)_n-(\omega)_n=n\delta_x\odot(\omega)_{n-1},\quad (\omega)^{(n)}-(\omega-\delta_x)^{(n)}=n\delta_x\odot(\omega)^{(n-1)}.\end{equation}

It follows from \eqref{xesa332a} and \eqref{fdeseas}, that for each $f^{(n)}$ from $\mathcal {F}^{(n)}(X)$ or $\mathcal {CF}^{(n)}(X)$, the functions $\langle(\omega)_n,f^{(n)}\rangle$ and  $\langle(\omega)^{(n)},f^{(n)}\rangle$ belong to $\mathcal P(M(X))$ or $\mathcal {CP}(M(X))$, respectively.
With an abuse of terminology, we will sometimes call these polynomials {\it falling and rising factorials}, respectively. The linear span of all falling (or all  rising) factorials forms the whole space $\mathcal P(M(X))$.

The falling factorials satisfy the recurrence relation
\begin{gather}
(\omega)_0=1,\notag\\
\langle(\omega)_{n+1},\xi^{\otimes(n+1)}\rangle=\langle(\omega)_{n},\xi^{\otimes n}\rangle\langle\omega,\xi\rangle-n\langle(\omega)_n,\xi^2\odot\xi^{\otimes(n-1)}\rangle,\quad\xi\in\mathcal F(X).
\label{fseraw4qa}
\end{gather}
From here a similar recurrence relation for the rising factorials follows.

The following generating functions uniquely characterize the falling and rising factorials:
\begin{align}\sum_{n=0}^\infty\frac1{n!}\,\langle(\omega)_n,\xi^{\otimes n}\rangle&=\exp\big[\langle\omega,\log(1+\xi)\rangle\big],\label{dw5w}\\
\sum_{n=0}^\infty\frac1{n!}\,\langle(\omega)^{(n)},\xi^{\otimes n}\rangle&=\exp\big[\langle\omega,-\log(1-\xi)\rangle\big],\quad\xi\in \mathcal F(X).\label{esa4q42}\end{align}
Both formulas \eqref{dw5w} and \eqref{esa4q42} are understood as equalities of formal power series, see Subsection~2.2 and Appendix in \cite{FKLO}. To obtain the  falling and rising factorials, one has to set in  formulas \eqref{dw5w}, \eqref{esa4q42} $\xi=z\psi$, where $\psi\in \mathcal F(X)$ and $z\in\mathbb F$, formally expand  in powers of $z$, and equate the coefficients by each $z^n$. 

For vector spaces $V$ and $W$, we denote by $\mathcal L(V,W)$ the space of linear operators from $V$ into $W$. We also denote   $\mathcal L(V):=\mathcal L(V,V)$.
For each $x\in X$, we define $\partial_x,D_x\in \mathcal L(\mathcal P(M(X)))$ by
\begin{align}
\partial_xp(\omega)&=\lim_{z\to 0}\frac{p(\omega+z\delta_x)-p(\omega)}z,\label{qdftyfe}\\
D_xp(\omega)&=p(\omega+\delta_x)-p(\omega)\label{r3q3q2sqg}
\end{align}
for each $p\in \mathcal P(M(X))$.  (The right-hand side of \eqref{qdftyfe} is just the derivative of $p$ in direction $\delta_x$.)
In particular, for each $f^{(n)}\in \mathcal F^{(n)}(X)$,
\begin{equation}\label{ydsrqqsqwdx}
\partial_x\langle\omega^{\otimes n},f^{(n)}\rangle=n\langle\omega^{\otimes(n-1)},f^{(n)}(x,\cdot)\rangle,\end{equation}
and by \eqref{s5wuwu},
\begin{equation}\label{rsa4q}
D_x\langle(\omega)_n,f^{(n)}\rangle=n\langle(\omega)_{n-1},f^{(n)}(x,\cdot)\rangle.
\end{equation}
It follows from \eqref{dw5w} that
\begin{equation}\label{cxrwq5w4}
 D_x=e^{\partial_x}-1=\sum_{n=1}^\infty\frac1{n!}\,\partial_x^n. \end{equation}
Note that for each $p\in\mathcal P(M(X))$ of degree $n$, we have $\partial_x^kp=0$ for all $k\ge n+1$, which justifies the use of the infinite series of powers of $\partial_x$ in \eqref{cxrwq5w4}.

\section{Stirling operators}\label{sa4t2q}

In Section \ref{rts54q43}, for $n,k\in\mathbb N$, $k\le n$, we defined the Stirling operators of the first kind, $\so(n,k)$, and the Stirling operators of the second kind, $\SO(n,k)$. Note that, using the notation~\eqref{dera4tq2y4}, we may write formulas~\eqref{cxta5wq5y43w} and~\eqref{cde5w35} as follows:
\begin{align}
\langle (\omega)_n,f^{(n)}\rangle&=\sum_{k=1}^n\langle\omega^{\otimes k},\so(n,k)f^{(n)}\rangle,\label{yrdqs}\\
\langle \omega^{\otimes n},f^{(n)}\rangle&=\sum_{k=1}^n\langle (\omega)_k,\SO(n,k)f^{(n)}\rangle.\label{fdr6ew456u4w}\end{align}
 Note that, for each polynomial $p$ on $M(X)$ of degree $n$, its representation as in formula~\eqref{tew53w5} is unique, hence formula~\eqref{yrdqs} indeed uniquely identifies the operators $\so(n,k)$.
Since $\so(n,n)$ is the identity operator, we similarly conclude that formula~\eqref{fdr6ew456u4w} indeed uniquely identifies the operators $\SO(n,k)$.

We further define the {\it unsigned Stirling operators of the first kind}, $\co(n,k)\linebreak\in\mathcal L(\mathcal F^{(n)}(X),\mathcal F^{(k)}(X))$, through the formula  
\begin{equation}
\langle(\omega)^{(n)},f^{(n)}\rangle=\sum_{k=1}^n\langle\omega^{\otimes k},\co(n,k)f^{(n)}\rangle.\label{drtrehgy}
\end{equation}

\begin{remark} By \eqref{ctsrea4w}, we obtain
\begin{equation}\label{xfdyf}
\so(n,k)=(-1)^{n-k}\co(n,k),\end{equation}
and
$$\langle\omega^{\otimes n},f^{(n)}\rangle=\sum_{k=1}^n\langle(\omega)^{(k)},(-1)^{n-k}\,\SO(n,k)f^{(n)}\rangle, \quad f^{(n)}\in\mathcal F^{(n)}(X).$$
\end{remark}

 Each space $\mathcal{CF}^{(n)}(X)$ ($n\in\mathbb N$) can be endowed with a natural locally convex topology, as a direct limit of the Banach spaces of $\mathbb F$-valued continuous symmetric functions with  supports in  compact subsets of $X^n$, equipped with the supremum norm. Note that the operators $\so(n,k)$, $\co(n,k)$, and $\SO(n,k)$ belong to $\mathcal L(\mathcal {CF}^{(n)}(X),\mathcal {CF}^{(k)}(X))$. Furthermore, it will follow from  Proposition~\ref{a4w5qwgb} below that these operators are  continuous.

The linear topological space $\mathcal{CF}^{(n)}(X)$ is Hausdorff and its dual space is  $M^{(n)}(X)$, see e.g.\ \cite[\S 29]{Bauer} or \cite[Chapter 7]{Simonnet}.
  Let $A\in\mathcal L(\mathcal {CF}^{(n)}(X),\mathcal {CF}^{(k)}(X))$ be a linear  continuous operator. Then, by e.g.\ \cite[Theorem 8.11.3]{NB}, the operator $A$ has the adjoint operator $A^*\in\mathcal L(M^{(k)}(X),M^{(n)}(X))$, satisfying $\langle \mu^{(k)},Af^{(n)}\rangle=\langle A^*\mu^{(k)},f^{(n)}\rangle$ for all $\mu^{(k)}\in M^{(k)}(X)$ and $f^{(n)}\in\mathcal {CF}^{(n)}(X)$. Formulas \eqref{yrdqs}--\eqref{drtrehgy} then imply that, for $\omega\in M(X)$, 
$$(\omega)_n=\sum_{k=1}^n\so(n,k)^*\omega^{\otimes k},\quad(\omega)^{(n)}=\sum_{k=1}^n\co(n,k)^*\omega^{\otimes k},\quad \omega^{\otimes n}=\sum_{k=1}^n\SO(n,k)^*(\omega)_k.$$

\begin{remark}\label{vyd6wq4a24}
 Let $A\in\mathcal L(\mathcal {CF}^{(n)}(X),\mathcal {CF}^{(k)}(X))$ be   a continuous linear operator, and let $A^*\in\mathcal L(M^{(k)}(X),M^{(n)}(X))$ be its adjoint. For each $\mu^{(k)}\in M^{(k)}(X)$, the measure $A^*\mu^{(k)}\in M^{(n)}(X)$ is completely identified by the integrals 
  $$\langle A^*\mu^{(k)},\xi^{\otimes n}\rangle=\langle \mu^{(k)},A\xi^{\otimes n}\rangle,$$
   where $\xi$ runs through $\mathcal {CF}(X)$. This implies that the linear operator $A$ is completely identified by $A\xi^{\otimes n}$ where $\xi$ runs through $\mathcal {CF}(X)$. (Equivalently, the linear span of the set $\{\xi^{\otimes n}\mid\xi\in\mathcal{CF}(X)\}$ is dense in $\mathcal{CF}^{(n)}(X)$.)
\end{remark}

\begin{remark}\label{rtsdqwsuqwds6qe}
For technical reasons, it is useful to define operators $\so(n,k)$, $\co(n,k)$, and $\SO(n,k)$ from $\mathcal L(\mathcal F^{(n)}(X),\mathcal F^{(k)}(X))$ to be zero 
if either $k>n\ge0$ or $k=0<n$ and to be equal to 1 if $k=n=0$. 
\end{remark}

We will now present explicit formulas for the action of the operators $\co(n,k)$ and $\SO(n,k)$. Let $k,n\in\mathbb N$, $k\le n$  and let $i_1,\dots,i_k\in\mathbb N$ be such that $i_1+\dots+i_k=n$. 
We define $\mathbb D^{(n)}_{i_1,\dots,i_k}\in\mathcal L(\mathcal F^{(n)}(X),\mathcal F^{(k)}(X))$ by
$$
(\mathbb D^{(n)}_{i_1,\dots,i_k}f^{(n)})(x_1,\dots,x_k):=P_k\big[f^{(n)}\big(\underbrace{x_1,\dots,x_1}_{\text{$i_1$ times}},\dots,\underbrace{x_k,\dots,x_k}_{\text{$i_k$ times}}\big)\big],\quad f^{(n)}\in\mathcal F^{(n)}(X). $$
Note that
\begin{equation}\label{cre4a543w6}
\mathbb D^{(n)}_{i_1,\dots,i_k}\xi^{\otimes n}=\xi^{i_1}\odot\dots\odot\xi^{i_k},\quad \xi\in\mathcal F(X).
\end{equation}
Obviously, we also have $\mathbb D^{(n)}_{i_1,\dots,i_k}\in\mathcal L(\mathcal {CF}^{(n)}(X),\mathcal {CF}^{(k)}(X))$.  Note that these operators are continuous.
Choosing $k=1$, gives the operator $\mathbb D^{(n)}:=\mathbb D^{(n)}_n$, 
\begin{equation}\label{cq5y42w}
(\mathbb D^{(n)}f^{(n)})(x)=f^{(n)}(x,\dots,x).
\end{equation}

\begin{remark}\label{se5wu563} If $X$ has only a single point, all the spaces $\mathcal F^{(n)}(X)$ can be identified with $\mathbb F$, in which case all operators $\mathbb D^{(n)}_{i_1,\dots,i_k}$ can be identified with number 1.
\end{remark}

\begin{remark}
 We will sometimes need an extension of $\mathbb D^{(n)}$ to $\mathcal F(X^n)$, which will act according to the same formula \eqref{cq5y42w}.
\end{remark}

Next, for $k,n\in\mathbb N$, $k\le n$, we denote by $\operatorname{UP}(n,k)$ the collection of all 
unordered  partitions of $\{1,\dots,n\}$ into $k$ non-empty parts.  For any $\lambda=\{\lambda_1,\dots,\lambda_k\}\in\operatorname{UP}(n,k)$, we define $\mathbb D^{(n)}_\lambda\in\mathcal L(\mathcal F(X^n),\mathcal F^{(k)}(X))$ as follows: 
let $f^{(n)}\in \mathcal F(X^n)$; in $f^{(n)}(x_1,\dots,x_n)$ replace each $x_j$ with $y_i$ where $j\in\lambda_i$;  symmetrize the obtained function of the $y_1,\dots,y_k$ variables; this function is $(\mathbb D^{(n)}_\lambda f^{(n)})(y_1,\dots,y_k)$. Note that, although in our definition of $\mathbb D^{(n)}_\lambda$ we used an order of the elements $\lambda_1,\dots,\lambda_k$ from 
$\lambda$, the obtained function does not depend on this order. 
 
\begin{remark} Let $\lambda=\{\lambda_1,\dots,\lambda_k\}\in\operatorname{UP}(n,k)$. We obviously have
\begin{equation}\label{cs5w5a}
\mathbb D^{(n)}_\lambda f^{(n)}=\mathbb D^{(n)}_{|\lambda_1|,\dots,|\lambda_k|}f^{(n)},\quad f^{(n)}\in\mathcal F^{(n)}(X).\end{equation} 
Here $|\lambda_l|$ denotes the number of elements of the set $\lambda_l$.  
On the other hand, the operator $\mathbb D^{(n)}_\lambda$ was defined for not necessarily symmetric functions, and we will use this fact below.
\end{remark}

\begin{proposition}\label{a4w5qwgb} For any $n,k\in\mathbb N$, $k\le n$, 
\begin{align}
\co(n,k)&=\frac{n!}{k!}\sum_{\substack{(i_1,\dots,i_k)\in\mathbb N^k\\i_1+\dots+i_k=n}}\frac1{i_1\dotsm i_k}\,\mathbb D^{(n)}_{i_1,\dots,i_k}\label{cx5waq4}\\
&=\sum_{\lambda=\{\lambda_1,\dots,\lambda_k\}\in \operatorname{UP}(n,k)}(|\lambda_1|-1)!\dotsm (|\lambda_k|-1)!\,\mathbb D^{(n)}_\lambda,\label{dts5ws53qq}\\
\SO(n,k)&=\frac{n!}{k!}\sum_{\substack{(i_1,\dots,i_k)\in\mathbb N^k\\i_1+\dots+i_k=n}}\frac1{i_1!\dotsm i_k!}\,\mathbb D^{(n)}_{i_1,\dots,i_k}\label{cta4waq2wed}\\
&=\sum_{\lambda\in \operatorname{UP}(n,k)}\mathbb D^{(n)}_\lambda.\label{xsa43wqa4q}
\end{align}
In particular, the operators $\so(n,k),\co(n,k),\SO(n,k)\in \mathcal L(\mathcal {CF}^{(n)}(X),\mathcal {CF}^{(k)}(X))$ are continuous.
\end{proposition}

\begin{proof}
It easily follows from  \eqref{esa4q42} that
\begin{align*}
\langle(\omega)^{(n)},\xi^{\otimes n}\rangle&=
\sum_{k=1}^n \frac{n!}{k!}\sum_{\substack{(i_1,\dots,i_k)\in\mathbb N^k\\i_1+\dots+i_k=n}}\frac1{i_1\dotsm i_k}\,\langle\omega,\xi^{i_1}\rangle\dotsm \langle\omega,\xi^{i_k}\rangle
\\
&=\sum_{k=1}^n \frac{n!}{k!}\sum_{\substack{(i_1,\dots,i_k)\in\mathbb N^k\\i_1+\dots+i_k=n}}\frac1{i_1\dotsm i_k}\,\langle\omega^{\otimes k},\xi^{i_1}\odot\dots\odot\xi^{i_k}\rangle.\end{align*}
This and \eqref{cre4a543w6} imply \eqref{cx5waq4}.  Next, formula \eqref{dw5w} implies
$$\sum_{n=0}^\infty\frac1{n!}\,\langle(\omega)_{n},(e^{\xi}-1)^{\otimes n}\rangle=\exp\big[\langle\omega,\xi\rangle\big],\quad\xi\in \mathcal F(X).$$
From here, analogously to the proof of \eqref{cx5waq4},  we derive  \eqref{cta4waq2wed}.
By  \eqref{cs5w5a}, \eqref{cx5waq4}, and \eqref{cta4waq2wed}, formulas \eqref{dts5ws53qq} and \eqref{xsa43wqa4q} easily follow. Finally, the statement about the continuity of the Stirling operators follows  from the continuity of $\mathbb D^{(n)}_{i_1,\dots,i_k}\in\mathcal L(\mathcal {CF}^{(n)}(X),\mathcal {CF}^{(k)}(X))$ and formulas  \eqref{xfdyf}, \eqref{cx5waq4}, and \eqref{cta4waq2wed}.
\end{proof}

\begin{remark}\label{tewu6534e} For $n,k\in\mathbb N$, $n\le k$, we denote by $\mathfrak S(n,k)$ the collection of all permutations $\pi\in\mathfrak S_n$ that have $k$ cycles. Each $\pi \in \mathfrak S(n,k)$ determines a partition $\lambda\in\operatorname{UP}(n,k)$ through the cycles of $\pi$. Furthermore, for each $\lambda=\{\lambda_1,\dots,\lambda_k\}\in \operatorname{UP}(n,k)$ there are $(|\lambda_1|-1)!\dotsm (|\lambda_k|-1)!$ permutations $\pi\in\mathfrak S(n,k)$ that give rise to $\lambda$. Thus, denoting $\mathbb D^{(n)}_\pi:=\mathbb D^{(n)}_\lambda$, we can rewrite formula \eqref{dts5ws53qq} as follows:
\begin{equation}\label{fstsdys}
\co(n,k)=\sum_{\pi\in\mathfrak S(n,k)}\mathbb D^{(n)}_\pi.\end{equation}
\end{remark}

\begin{remark}
For $A\in\mathcal B_0(X)$, let $\chi_{A^n}$ denote the indicator function of the set $A^n$. Then, by \eqref{xfdyf}, \eqref{xsa43wqa4q}, and \eqref{fstsdys},
\begin{equation}\label{tsa4waAA}
\co(n,k)\chi_{A^n}=c(n,k)\chi_{A^k}\,,\quad \so(n,k)\chi_{A^n}=s(n,k)\chi_{A^k}\,,\quad \SO(n,k)\chi_{A^n}=S(n,k)\chi_{A^k}\,.\end{equation}
Here $c(n,k)$, $s(n,k)$, and $S(n,k)$ denote the classical unsigned and signed Stirling numbers of the first kind and the classical Stirling numbers of the second kind, respectively. (Recall that $c(n,k)$ is equal to the number of permutations in $\mathfrak S(n,k)$, and $S(n,k)$ is equal to the number of partitions in $\operatorname{UP}(n,k)$.) If $X$ has only a single point, then by \eqref{tsa4waAA} and Remark~\ref{se5wu563}, the linear operators $\co(n,k)$, $\so(n,k)$, $\SO(n,k)$ can be identified with the  numbers $c(n,k)$, $s(n,k)$, $S(n,k)$, respectively.
\end{remark}

\begin{remark} Let $\mathbb P(n,k)$ denote the collection of all $(r_1,\dots,r_{n-k+1})\in\mathbb N_0^{n-k+1}$ such that $r_1+r_2+\dots+r_{n-k+1}=k$ and $r_1+2r_2+\dots+(n-k+1)r_{n-k+1}=n$. For each $(r_1,r_2,\dots,r_{n-k+1})\in\mathbb P(n,k)$, we denote
$$\alpha(r_1,\dots,r_{n-k+1}):=\frac{n!}{r_1!\,r_2!\dotsm r_{n-k+1}!\,(1!)^{r_1}(2!)^{r_2}\dotsm((n-k+1)!)^{r_{n-k+1}}},$$
which is the number of partitions from $\operatorname{UP}(n,k)$ that contains $r_1$ parts that have exactly  one element, $r_2$ parts that contain exactly two elements, and so on. By \eqref{xsa43wqa4q},
\begin{multline*}
\SO(n,k)=\sum_{(r_1,\dots,r_{n-k+1})\in\mathbb P(n,k)}\alpha(r_1,\dots,r_{n-k+1})\\
\times P_k\big((\mathbb D^{(1)})^{\otimes r_1}\otimes (\mathbb D^{(2)})^{\otimes r_2}\otimes\dots\otimes (\mathbb D^{(n-k+1)})^{\otimes r_{n-k+1}}\big).\end{multline*}
This formula resembles the definition of the {\it (partial) exponential Bell polynomials} \cite{Bell}. The latter multivariate polynomials  are defined by
\begin{equation}\label{rw5w2u53}
B_{n,k}(z_1,\dots,z_{n-k+1})=\sum_{(r_1,\dots,r_{n-k+1})\in\mathbb P(n,k)}\alpha(r_1,\dots,r_{n-k+1})\,z_1^{r_1}z_2^{r_2}\dotsm z_{n-k+1}^{r_{n-k+1}}\end{equation}
for $(z_1,\dots,z_{n-k+1})\in\mathbb F^{n-k+1}$. In \cite{Schreiber}, these polynomials were called the {\it multivariate Stirling polynomials of the second kind}. Furthermore, by \eqref{dts5ws53qq}, we obtain
\begin{multline*}
\so(n,k)=\sum_{(r_1,\dots,r_{n-k+1})\in\mathbb P(n,k)}\beta(r_1,\dots,r_{n-k+1})\\
\times P_k\big((\mathbb D^{(1)})^{\otimes r_1}\otimes (\mathbb D^{(2)})^{\otimes r_2}\otimes\dots\otimes (\mathbb D^{(n-k+1)})^{\otimes r_{n-k+1}}\big),\end{multline*}
where
$$\beta(r_1,\dots,r_{n-k+1}):=\frac{n!}{r_1!\,r_2!\dotsm r_{n-k+1}!\,1^{r_1}\,2^{r_2}\dotsm (n-k+1)^{r_{n-k+1}}}.$$
Note, however, that if we defined a new sequence of multivariate polynomials by replacing in formula~\eqref{rw5w2u53}, the $\alpha$ coefficients with the $\beta$ coefficients, then these would differ from the multivariate Stirling polynomials of the first kind as defined in \cite{Schreiber}. 
\end{remark}

We will now discuss recurrence relations satisfied by the Stirling operators. For linear operators $L_i\in\mathcal L(\mathcal F(X^{n_i}),\mathcal F(X^{k_i}))$,
 $i=1,2$, we can define their tensor product $L_1\otimes L_2\in\mathcal L(\mathcal F(X^{n_1+n_2}),\mathcal F(X^{k_1+k_2}))$. 

\begin{proposition}\label{cdrttes} Let $\mathbf 1^{(n)}$ denote the identity operator on $\mathcal F(X^n)$. We have:
\begin{gather}
\SO(n,n)=\so(n,n)=\mathbf 1^{(n)},\label{rses3waw2}\\
\SO(n+1,k)=P_k\big(\mathbf 1^{(1)}\otimes\SO(n,k-1)\big)+kP_k\big((\mathbb D^{(2)}\otimes \mathbf 1^{(k-1)})(\mathbf 1^{(1)}\otimes\SO(n,k))\big), \label{crtw5w}\\
\so(n+1,k)=P_k\big(\mathbf 1^{(1)}\otimes\so(n,k-1)\big)-n\so(n,k)P_n\big(\mathbb D^{(2)}\otimes\mathbf 1^{(n-1)}\big).\label{dw4w34}
\end{gather}
for $n\in\mathbb N$ and $k=1,\dots, n$.
\end{proposition}

\begin{remark}
Formulas \eqref{xfdyf}, \eqref{rses3waw2}, and \eqref{dw4w34} immediately imply a recurrence relation for $\co(n,k)$.
\end{remark}

\begin{remark} In the case of a single-point underlying space $X$, formulas  \eqref{crtw5w}, \eqref{dw4w34} become  the classical recurrence formulas $S(n+1,k)=S(n,k-1)+kS(n,k)$ and $s(n+1,k)=s(n,k-1)-ns(n,k)$.
\end{remark}

\begin{proof}[Proof of Proposition \ref{cdrttes}] We denote by $\operatorname{OP}(n,k)$ the colection of all ordered partitions  of $\{1,\dots,n\}$ into $k$ non-empty parts. For each $\lambda=(\lambda_1,\dots,\lambda_k)\in \operatorname{OP}(n,k)$, we define the operator $\mathbb D^{(n)}_\lambda\in\mathcal L(\mathcal F(X^n),\mathcal F(X^k))$ analogously to the case when $\lambda$ was an unordered partition but without the symmetrization of the obtained function in the end. By \eqref{xsa43wqa4q},
$\SO(n,k)=\frac1{k!}\sum_{\lambda\in \operatorname{OP}(n,k)}\mathbb D^{(n)}_\lambda$. Note also that if a function $g^{(k)}:X^k\to\R$ is symmetric in its $x_2,\dots,x_k$ variables, then
$$(P_kg^{(k)})(x_1,\dots,x_k)=\frac1k\sum_{i=1}^k g^{(k)}(x_2,\dots,x_{i-1},x_1, x_i,\dots,x_k).$$
Therefore,
\begin{align*}
&P_k\big(\mathbf 1^{(1)}\otimes\SO(n,k-1)\big)=\frac1{k!}\sum_{\lambda=(\lambda_1,\dots,\lambda_k)\in\operatorname{OP}(n+1,k):\  \{1\}=\lambda_i\text{ for some }i}\mathbb D^{(n)}_\lambda,\\
&kP_k\big((\mathbb D^{(2)}\otimes \mathbf 1^{(k-1)})(\mathbf 1^{(1)}\otimes\SO(n,k))\big)=\frac1{k!}\sum_{\lambda=(\lambda_1,\dots,\lambda_k)\in\operatorname{OP}(n+1,k):\  \lambda_i\ne\{1\}\text{ for all }i}\mathbb D^{(n)}_\lambda,
\end{align*}
which implies \eqref{crtw5w}. While formula \eqref{dw4w34} could be derived from \eqref{xfdyf} and \eqref{fstsdys}, it is in fact a direct consequence of the recurrence formula \eqref{fseraw4qa}.
\end{proof}

\begin{remark}\label{ts5as2}
   Let $A\in\mathcal L(\mathcal {CF}^{(n)}(X),\mathcal {CF}^{(k)}(X))$ be   a continuous linear operator, and let $A^*\in\mathcal L(M^{(k)}(X),M^{(n)}(X))$ be its adjoint. Each $f^{(n)}\in\mathcal F^{(n)}(X)$ determines a linear functional 
\begin{equation}\label{tew53wq}
M^{(k)}(X)\ni\mu^{(k)}\mapsto \langle A^*\mu^{(k)},f^{(n)}\rangle.\end{equation}
 Assume that, for each $f^{(n)}\in\mathcal F^{(n)}(X)$, there exists a function $g^{(k)}\in\mathcal F^{(k)}(X)$ that allows to represent the functional 
\eqref{tew53wq} as $\langle\mu^{(k)},g^{(k)}\rangle$. (The function $g^{(k)}$ is then obviously unique.)
Then, we can extend $A$ to  the linear operator from  $\mathcal L(\mathcal F^{(n)}(X),\mathcal F^{(k)}(X))$ defined by $Af^{(n)}:=g^{(k)}$. We may call it the {\it canonical extension} of $A$. In particular, in view of Remark~\ref{vyd6wq4a24}, the canonical extension is completely identified by $A\xi^{\otimes n}$ where $\xi$ runs through $\mathcal {CF}(X)$, or through the larger set $\mathcal {F}(X)$. Finally, we note that canonical extensions of the Stirling operators $\so(n,k),\SO(n,k)\in \mathcal L(\mathcal {CF}^{(n)}(X),\mathcal {CF}^{(k)}(X))$  exist and are equal to the Stirling operators  $\so(n,k)$ and $\SO(n,k)$ from $\mathcal L(\mathcal {F}^{(n)}(X),\mathcal {F}^{(k)}(X))$. \end{remark}

The following proposition gives the explicit form of the  generating functions of the Stirling operators of the first and second kind. In view of Remark~\ref{ts5as2}, these generating functions uniquely characterize the Stirling operators. 

\begin{proposition}\label{cesseas} We have, for each $k\in\mathbb N$,
\begin{align}
\sum_{n=k}^\infty\frac1{n!}\,\SO(n,k)\xi^{\otimes n}&=\frac1{k!}\,(e^\xi-1)^{\otimes k},\label{xtsaer5aw5}\\
\sum_{n=k}^\infty\frac1{n!}\,\so(n,k)\xi^{\otimes n}&=\frac1{k!}\,(\log(1+\xi))^{\otimes k},\quad\xi\in\mathcal F(X).\label{dre6ew6ww}
\end{align}
\end{proposition}

\begin{remark} In the case of a single-point underlying space $X$, formula \eqref{xtsaer5aw5} becomes 
the classical formula for the generating function of the Stirling numbers of the second kind: $\sum_{n=k}^\infty\frac{\xi^n}{n!}\,S(n,k)=\frac1{k!}(e^\xi-1)^k$ ($\xi\in\mathbb F$). The respective interpretation of  \eqref{dre6ew6ww} is similar.
\end{remark}

\begin{remark}
Formulas \eqref{xfdyf} and \eqref{dre6ew6ww} yield the generating function of the operators $\co(n,k)$, which is equal to $\frac1{k!}\,(-\log(1-\xi))^{\otimes k}$.
\end{remark}

\begin{proof}[Proof of Proposition \ref{cesseas}] We have
$$\frac1{k!}\,(e^\xi-1)^{\otimes k}=\frac1{k!}\,\bigg(\sum_{i=1}^\infty\frac{\xi^i}{i!}\bigg)^{\otimes k}=\frac1{k!}\sum_{n=k}^\infty \sum_{\substack{(i_1,\dots,i_k)\in\mathbb N^k\\i_1+\dots+i_k=n}}\frac1{i_1!\dotsm i_k!}\,\xi^{i_1}\odot\dots\odot\xi^{i_k}.$$
Hence, \eqref{xtsaer5aw5} follows from \eqref{cta4waq2wed}. The proof of \eqref{dre6ew6ww} is analogous.
\end{proof}

\begin{corollary} For $\omega\in M(X)$ and $\xi\in\mathcal F(X)$,
\begin{align}
\langle\omega^{\otimes k},\SO(n,k)\xi^{\otimes n}\rangle&=\frac1{k!}\,\frac{d^n}{dz^n}\Big|_{z=0}\langle\omega, e^{z\xi}-1\rangle^n,\label{vctrs5e}\\\langle\omega^{\otimes k},\so(n,k)\xi^{\otimes n}\rangle&=\frac1{k!}\,\frac{d^n}{dz^n}\Big|_{z=0}\langle\omega, \log(1+z\xi)\rangle^n. \label{dxerts}\end{align}
\end{corollary}

\begin{proof} By \eqref{xtsaer5aw5}, for $\omega\in M(X)$, $\xi\in\mathcal F(X)$, and $z\in\mathbb F$,
$$\sum_{n=k}^\infty\frac{z^n}{n!}\,\langle\omega^{\otimes k},\SO(n,k)\xi^{\otimes n}\rangle=\langle\omega^{\otimes k},\frac1{k!}\,(e^{z\xi}-1)^{\otimes k}\rangle=\frac1{k!}\langle\omega, e^{z\xi}-1\rangle^k,$$
which is an equality of formal power series in $z$. This immediately implies \eqref{vctrs5e}. The proof of \eqref{dxerts} is analogous. 
\end{proof}

The following proposition shows that the classical orthogonality identities for the Stirling numbers (e.g.\ \cite[Section~12.2]{QG}) admit a generalization to our setting.

\begin{proposition}\label{fye6e}
For any $i,n\in\mathbb N$,
\begin{equation}\label{yr6e3}
\sum_{k=1}^n\so(k,i)\SO(n,k)=\sum_{k=1}^n\SO(k,i)\so(n,k)=\delta_{ni}\mathbf 1^{(i)},
\end{equation}
where $\delta_{ni}$ is the Kronecker symbol.
\end{proposition}

\begin{proof} The proof is essentially the same as in the classical case. For the reader's convenience, we still present it. Formula \eqref{yr6e3} obviously holds when $i>n$. Let $i\le n$. 
For each $f^{(n)}\in\mathcal F^{(n)}(X)$, we get
\begin{align*}
\langle\omega^{\otimes n},f^{(n)}\rangle&=\sum_{k=1}^n \langle(\omega)_k,\SO(n,k)f^{(n)}\rangle\\
&=\sum_{k=1}^n\sum_{i=1}^k\langle\omega^{\otimes i},\so(k,i)\SO(n,k)f^{(n)}\rangle\\
&=\sum_{i=1}^n\big\langle\omega^{\otimes i},\sum_{k=i}^n \so(k,i)\SO(n,k)f^{(n)}\big\rangle\\&=\sum_{i=1}^n\big\langle\omega^{\otimes i},\sum_{k=1}^n \so(k,i)\SO(n,k)f^{(n)}\big\rangle,
\end{align*}
which proves the first equality in \eqref{yr6e3}. The proof of the second equality  is similar.
\end{proof}

We finish this section with a short discussion of Lah operators. For $n,k\in\mathbb N$, $k\le n$, we define the {\it Lah operator} $\LO(n,k)\in\mathcal L(\mathcal F^{(n)}(X),\mathcal F^{(k)}(X))$ by
\begin{equation}\label{xtesa5y}
\langle(\omega)^{(n)},f^{(n)}\rangle=\sum_{k=1}^n\langle(\omega)_k\,,\LO(n,k)f^{(n)}\rangle, \quad f^{(n)}\in\mathcal F^{(n)}(X).\end{equation}
In view of \eqref{ctsrea4w}, we equivalently have
$$\langle(\omega)_n,f^{(n)}\rangle=\sum_{k=1}^n\langle(\omega)^{(k)},(-1)^{n-k}\,\LO(n,k)f^{(n)}\rangle, \quad f^{(n)}\in\mathcal F^{(n)}(X).$$

\begin{proposition} (i) For $n,k\in\mathbb N$, $k\le n$, we have
\begin{align}
\LO(n,k)&=\sum_{i=k}^n\SO(i,k)\co(n,i)\label{ra4q4}\\
&=\frac{n!}{k!}\sum_{\substack{(i_1,\dots,i_k)\in\mathbb N^k\\i_1+\dots+i_k=n}}\mathbb D^{(n)}_{i_1,\dots,i_k}\label{dcutwr}\\
&=\sum_{\lambda=\{\lambda_1,\dots,\lambda_k\}\in \operatorname{UP}(n,k)}|\lambda_1|!\,\dotsm |\lambda_k|!\,\,\mathbb D^{(n)}_\lambda.\label{CWD}
\end{align}

(ii) For each $k\in\mathbb N$,
\begin{equation}\label{vtsw5w3}
\sum_{n=k}^\infty\frac1{n!}\,\LO(n,k)\xi^{\otimes n}=\frac{1}{k!}\,\bigg(\frac{\xi}{1-\xi}\bigg)^{\otimes k},\quad\xi\in\mathcal F(X).\end{equation}
\end{proposition}

\begin{proof}
Equality \eqref{ra4q4} follows immediately from \eqref{fdr6ew456u4w}, \eqref{drtrehgy},   and \eqref{xtesa5y}. By \eqref{cx5waq4}, \eqref{cta4waq2wed}, and \eqref{ra4q4}, we obtain
\begin{align}
\LO(n,k)&=\sum_{j=k}^n\frac{j!}{k!}\sum_{\substack{(i_1,\dots,i_k)\in\mathbb N^k\\i_1+\dots+i_k=j}}\frac1{i_1!\dotsm i_k!}\,\mathbb D^{(j)}_{i_1,\dots,i_k}\,\frac{n!}{j!}\,\sum_{\substack{(l_1,\dots,l_j)\in\mathbb N^j\\l_1+\dots+l_j=n}}\frac1{l_1\dotsm l_j}\,\mathbb D^{(n)}_{l_1,\dots,l_j}\notag\\
&=\frac{n!}{k!}\sum_{j=k}^n \sum_{\substack{(i_1,\dots,i_k)\in\mathbb N^k\\i_1+\dots+i_k=j}}\sum_{\substack{(l_1,\dots,l_j)\in\mathbb N^j\\l_1+\dots+l_j=n}}\frac1{i_1!\dotsm i_k!}\,\frac1{l_1\dotsm l_j}\,\mathbb D^{(j)}_{i_1,\dots,i_k}\,
\mathbb D^{(n)}_{l_1,\dots,l_j}\notag\\
&=\frac{n!}{k!}\,\sum_{\substack{(l_1,\dots,l_k)\in\mathbb N^k\\l_1+\dots+l_k=n}} r(l_1)\dotsm r(l_k)\mathbb D^{(n)}_{l_1,\dots,l_k}.\label{seqw54q327}
\end{align}
Here, for $l\in\mathbb N$,
\begin{equation}
r(l):=\sum_{m=1}^l\frac1{m!}\,\sum_{\substack{(i_1,\dots,i_m)\in\mathbb N^m\\i_1+\dots+i_m=l}}\frac1{i_1\dotsm i_m}=\frac1{l!}\sum_{m=1}^lc(l,m)=\frac{l!}{l!}=1,\label{erq4332q}
\end{equation}
where we used Remark \ref{tewu6534e}. Formulas \eqref{seqw54q327} and \eqref{erq4332q} imply \eqref{dcutwr}. Next, \eqref{cs5w5a} and \eqref{dcutwr} imply \eqref{CWD}. Finally, formula \eqref{vtsw5w3} easily follows from \eqref{dcutwr}  by analogy with the proof of Proposition~\ref{cesseas}. 
\end{proof}

\begin{corollary}
For any $i,n\in\mathbb N$,
$$\sum_{k=1}^n(-1)^{n-k}\,\mathbf L(k,i)\,\mathbf L(n,k)=\delta_{ni}\mathbf 1^{(i)}.
$$
\end{corollary}

\begin{proof}
Immediate by \eqref{xfdyf}, \eqref{ra4q4}, and Proposition~\ref{fye6e}.
\end{proof}

\section{Euler's formula for the Stirling operators of the~second kind}\label{few53q}

\subsection{Euler's formula}

We denote by $\Gamma_0(X)$ the subset of the configuration 
space $\Gamma(X)$ consisting of all finite (possibly empty) configurations. Thus, each element $\eta\in \Gamma_0(X)$  is understood as both a finite subset $\eta=\{x_1,\dots,x_n\}$ of $X$ and a finite measure $\delta_{x_1}+\dots+\delta_{x_n}$. (Note that the empty set is associated with the zero measure.)
We now need an extension of $\Gamma_0(X)$, the set of multiple finite configurations in $X$, denoted by $\ddot\Gamma_0(X)$. Each element $\eta\in \ddot\Gamma_0(X)$  is understood as both a finite multiset $\eta=[x_1,\dots,x_n]$ and the corresponding finite measure $\delta_{x_1}+\dots+\delta_{x_n}$.
As a subset of $M(X)$, $\ddot\Gamma_0(X)$ is the set of all positive finite integer-valued measures on $X$. 

The following theorem gives an infinite dimensional counterpart of Euler's formula for $S(n,k)$, cf.\ \cite[Section 9.1]{QG}.

\begin{theorem}\label{rsdtresw45w}
Let $p\in\mathcal P(M(X))$ be a polynomial of degree $n$. Then
\begin{equation}\label{x4tw5wq}
p(\omega)=\sum_{k=0}^n\langle(\omega)_k,g^{(k)}\rangle,
\end{equation} where $g^{(0)}=p(0)$ and, for each $k=1,\dots,n$,
\begin{align}
g^{(k)}(x_1,\dots,x_k)&=\frac1{k!}\,(D_{x_1}\dotsm D_{x_k}\,p)(0)\label{cddssw4}\\
&=\frac{(-1)^k}{k!}\sum_{\eta\subseteq[x_1,\dots,x_k]}(-1)^{|\eta|}p(\eta).\label{rtdesse}\end{align}
Here $|\eta|$ is the cardinality of the multiset $\eta$, equivalently $|\eta|=\eta(X)$.

In particular, for each $n\in\mathbb N$, $k=1,\dots,n$, and 
$f^{(n)}\in\mathcal F^{(n)}(X)$,
\begin{equation}\label{fxqdqs}
(\SO(n,k)f^{(n)})(x_1,\dots,x_k)=\frac{(-1)^k}{k!}\sum_{\eta\subseteq[x_1,\dots,x_k]}(-1)^{|\eta|}\langle\eta^{\otimes n},f^{(n)}\rangle.\end{equation}
\end{theorem}

\begin{proof}  Formula \eqref{cddssw4} follows from \eqref{rsa4q} and the polynomial expansion theorem \cite[Proposition 4.6]{FKLO}. Note that, in \cite{FKLO}, this theorem was proved under slightly different assumptions. Nevertheless, an easy check shows that it remains true in our setting.

 By \eqref{r3q3q2sqg} and the induction, we easily conclude that
\begin{equation}\label{dt5w56uw}
(D_{x_1}\dotsm D_{x_k}\,p)(\omega)=(-1)^k\sum_{\eta\subseteq[x_1,\dots,x_k]}(-1)^{|\eta|}p(\omega+\eta).\end{equation}
By \eqref{cddssw4} and \eqref{dt5w56uw}, we get \eqref{rtdesse}.
Setting $p(\omega)=\langle\omega^{\otimes n},f^{(n)}\rangle$ into formulas \eqref{x4tw5wq}, \eqref{rtdesse}, and using \eqref{fdr6ew456u4w}, we obtain 
  \eqref{fxqdqs}.
\end{proof}

\begin{remark}\label{tdr652r6} In view of Remark \ref{rtsdqwsuqwds6qe}, formula \eqref{fxqdqs} holds, in fact, for all $n,k\in\mathbb N_0$.
\end{remark}

\begin{remark}
Formulas \eqref{x4tw5wq}, \eqref{rtdesse}  imply that each polynomial $p\in\mathcal P(M(X))$ is uniquely determined by its values on $\ddot\Gamma_0(X)$. 
\end{remark}

\begin{remark} In the case of a single-point space $X$, formula \eqref{fxqdqs} becomes  the classical Euler's formula
$S(n,k)=\frac{(-1)^k}{k!}\sum_{l=1}^k(-1)^l\binom k l l^n$.
\end{remark}

\subsection{$\mathcal K$-transform}\label{fdsters5w5}

Let us now briefly discuss how Theorem \ref{rsdtresw45w} is related to the
theory of point processes. 

Denote by $\mathcal F(\ddot\Gamma_0(X))$ the algebraic direct sum of the $\mathcal F^{(n)}(X)$ spaces, $n\in\mathbb N_0$. 
Thus, each $f\in \mathcal F(\ddot\Gamma_0(X))$ is of the form $f=(f^{(n)})_{n=0}^\infty$ with $f^{(n)}\in \mathcal F^{(n)}(X)$ and, for some $N\in\mathbb N$, $f^{(n)}=0$ for all $n\ge N$. We may identify $\mathcal F(\ddot\Gamma_0(X))$ with a class of functions on $\ddot\Gamma_0(X)$. Indeed, for each $f=(f^{(n)})_{n=0}^\infty$, define $f(\varnothing):=f^{(0)}$ and $f([x_1,\dots,x_n]):=f^{(n)}(x_1,\dots,x_n)$.
Below, with an abuse of notation, we will use both interpretations of elements of $\mathcal F(\ddot\Gamma_0(X))$.

Similarly to \cite[Subsection~3.1]{KunaInfusino} (see also \cite{Len75b}), we define a bijective map 
$\mathcal K:\mathcal F(\ddot\Gamma_0(X))\to\mathcal P(M(X))$ by 
$$ (\mathcal Kf)(\omega)=\sum_{n=0}^\infty\left\langle \binom \omega n,f^{(n)}\right\rangle,$$
the sum being in fact finite. In particular,  by formula \eqref{reaq43q42}, for 
$\omega=\gamma\in\Gamma(X)$, we get
\begin{equation}\label{drtw5wsd}
(\mathcal Kf)(\gamma)=\sum_{\eta\subseteq\gamma,\ \eta\in\Gamma_0(X)}f(\eta).\end{equation}
Theorem \ref{rsdtresw45w} implies that the inverse map $\mathcal K^{-1}:\mathcal P(M(X))\to\mathcal F(\ddot\Gamma_0(X))$ 
 is given by
\begin{equation}\label{5w45q43q}
(\mathcal K^{-1}p)(\eta)=\sum_{\sigma\subseteq\eta}(-1)^{|\eta|-|\sigma|}p(\sigma).\end{equation}

In the theory of point processes, one considers a slightly different map, denoted by $K$. Let $\mathcal F(\Gamma_0(X))$ denote the class of function on $\Gamma_0(X)$ obtained as restrictions of functions from $\mathcal F(\ddot\Gamma_0(X))$. (Note that, for each $f\in \mathcal F(\Gamma_0(X))$, there are infinitely many functions from $\mathcal F(\ddot\Gamma_0(X))$ whose restriction to $\Gamma_0(X)$ coincides with $f$.) Similarly, let $\mathcal P(\Gamma(X))$ denote the class of functions on $\Gamma(X)$ obtained as restrictions of polynomials from $\mathcal P(M(X))$.  One defines a bijective map $K:\mathcal F(\Gamma_0(X))\to\mathcal P(\Gamma(X))$ by formula \eqref{drtw5wsd} in which $\mathcal K$ is replaced by $K$ and $f\in \mathcal F(\Gamma_0(X))$. The inverse operator, $K^{-1}$, is then given by formula \eqref{5w45q43q} in which $\mathcal K$ replaced by $K$ and $p\in \mathcal P(\Gamma(X))$.  Note that, in this case, formula \eqref{5w45q43q} is just a straightforward application of the inclusion-exclusion principle. 

The main reason for introducing the $K$-transform in the theory of point processes is that, for a point process on $X$ (equivalently a probability measure on $\Gamma(X)$), the measure $\theta$ on $\Gamma_0(X)$ defined by $\mathbb E(Kf)=\int_{\Gamma_0(X)}f\,d\theta$ is called the {\it correlation measure} of the point process, and (under certain weak assumptions) $\theta$ uniquely determines the point process. See e.g.\ \cite{KK,Len1,Len75b} for details. 

One defines a binary operation $\star$ on $\mathcal F(\Gamma_0(X))$ so that, for any $f,g\in\mathcal F(\Gamma_0(X))$, one has $\big(K(f\star g)\big)(\gamma)=(Kf)(\gamma)(Kg)(\gamma)$. An easy calculation show that 
\begin{equation}\label{sea453q}
(f\star g)(\eta)=\sum_{ \substack{\sigma_1,\sigma_2,\sigma_3\in\Gamma_0(X) \\ \sigma_1+\sigma_2+\sigma_3=\eta} }f(\sigma_1+\sigma_2)g(\sigma_2+\sigma_3),\quad\eta\in\Gamma_0(X),\end{equation}
see \cite{KK}. We now extend the binary operation $\star$ to $\mathcal F(\ddot\Gamma_0(X))$ by requiring that
\begin{equation}\label{stw6us}
\big(\mathcal K(f\star g)\big)(\omega)=(\mathcal Kf)(\omega)(\mathcal Kg)(\omega)\quad\text{for all } \omega\in M(X).
\end{equation}

\begin{remark} In the theory of point processes, the $K$-transform is often thought of as a counterpart of the Fourier transform, see e.g.\ \cite{KK}.  Hence, in view of formula \eqref{stw6us}, it is natural to interpret the binary operation $\star$ as a convolution of functions from $\mathcal F(\ddot\Gamma_0(X))$.
\end{remark}

\begin{proposition}\label{sweawqa}
For any $f,g\in \mathcal F(\ddot\Gamma_0(X))$, the $\star$ product of $f$ and $g$ is given by formula \eqref{sea453q} in which $\Gamma_0(X)$ is replaced by $\ddot\Gamma_0(X)$. 
\end{proposition}

\begin{proof} The statement can be immediately concluded from (the proof of)\cite[Proposition~3.4]{KunaInfusino}. For the reader's convenience, we will now outline an (alternative) proof of it. 
Formula \eqref{fseraw4qa} implies that, for any $f^{(n)}\in\mathcal F^{(n)}(X)$ and $\xi\in \mathcal F(X)$,
\begin{equation}\label{fsre5w5}
f^{(n)}\star\xi=(n+1)f^{(n)}\odot\xi+N(\xi)f^{(n)}.\end{equation}
Here $N(\xi)f^{(n)}\in\mathcal F^{(n)}(X)$ is defined by
\begin{equation}\label{cresa5}
(N(\xi)f^{(n)})(x_1,\dots,x_n):=f^{(n)}(x_1,\dots,x_n)(\xi(x_1)+\dots+\xi(x_n)).\end{equation}

It suffices to prove formula \eqref{sea453q} for $f=f^{(n)}\in\mathcal F^{(n)}(X)$ and $g=\xi^{\otimes m}\in\mathcal F^{(m)}(X)$.  For $m=1$, the result follows directly from formula~\eqref{fsre5w5}. Assume that it holds for $m$. By \eqref{fsre5w5},
$$f^{(n)}\star\xi^{\otimes (m+1)}=\frac1{m+1}\,\big((f^{(n)}\star\xi^{\otimes m})\star\xi-f^{(n)}\star (N(\xi)\xi^{\otimes m})\big),$$
then one uses the induction assumption applied to $f^{(n)}\star\xi^{\otimes m}$ and $f^{(n)}\star (N(\xi)\xi^{\otimes m})$, and then formula~\eqref{sea453q} again. \end{proof}

\begin{remark}
Let us consider the one-dimensional counterpart of Proposition \ref{sweawqa}, i.e., the case when $X$ has a single point. Consider $\mathcal F(\mathbb N_0)$, the space of  $\mathbb F$-valued functions on $\mathbb N_0$ (sequences) that vanish at all but finitely many points, and denote by $\mathcal P(\mathbb F)$ the space of polynomials on $\mathbb F$.  Consider the bijective map $\mathcal K:\mathcal F(\mathbb N_0)\to\mathcal P(\mathbb F)$ given by $( Ka)(z):=\sum_{n=0}^\infty a(n)(z)_n$\,, and define a binary operation $\star$ on $\mathcal F(\mathbb N_0)$ by requiring $\big(K(a\star b)\big)(z)=(Ka)(z)(Kb)(z)$. Then, by  Proposition \ref{sweawqa}, 
$$(a\star b)(n)=\sum_{ \substack{ i,j,k\in\mathbb{N}_0 \\ i+j+k=n} }\binom n {i\ j\ k} a(i+j)b(j+k).$$
\end{remark}

\section{Identities for the Stirling operators}\label{tes5w}

We will now discuss a few identities satisfied by the Stirling operators. All of them will yield classical identities for the Stirling numbers when the space $X$ has a single point.

\begin{proposition}[Infinite dimensional Olson's identity]\label{dtrsews3a}
 Let $m,n,i\in\mathbb N$ and denote $l:=m+n$. Then, for each $f^{(l)}\in\mathcal F^{(l)}(X)$, we have
\begin{align}
&\bigg(\sum_{k=1}^n \SO(k+m,i)P_{m+k}\big(\mathbf 1^{(m)}\otimes \so(n,k)\big)f^{(l)}\bigg)(x_1,\dots,x_i)\notag\\
&\quad =\frac{(-1)^i}{i!}\sum_{\eta\subseteq[x_1,\dots,x_i],\, |\eta|\ge n}(-1)^{|\eta|}\langle(\eta)_n\odot \eta^{\otimes m},f^{(l)}\rangle.\label{xtsw5w}
\end{align}
If either $i<n$ or $i>l$, the right hand side of \eqref{xtsw5w} is equal to zero, and if $i=n$, it is equal to
\begin{equation}\label{cxsesa53w}
\big\langle(\delta_{x_1}+\dots+\delta_{x_n})^{\otimes m},f^{(l)}(x_1,\dots,x_n,\cdot)\big\rangle=\sum_{j_1=1}^n\dotsm\sum_{j_m=1}^n f^{(l)}(x_1,\dots,x_n,x_{j_1},\dots,x_{j_m}).\end{equation}
\end{proposition}

\begin{proof} We first note that formula \eqref{reaq43q42} and Remark \ref{c4q4q} remain true when $\gamma\in\ddot\Gamma_0(X)$.

Just to simplify notations, we assume that $f^{(l)}=\varphi^{\otimes l}$ with $\varphi\in\mathcal F(X)$.
Using Theorem~\ref{rsdtresw45w} (see also Remark~\ref{tdr652r6}), and Proposition \ref{fye6e}, we get
\begin{align}
&\bigg(\sum_{k=1}^n \SO(k+m,i)P_{m+k}\big(\mathbf 1^{(m)}\otimes \so(n,k)\big)\varphi^{\otimes l}\bigg)(x_1,\dots,x_i)\notag\\
&\quad=\bigg(\sum_{k=1}^n \SO(k+m,i)\big(\varphi^{\otimes m}\odot(\so(n,k)\varphi^{\otimes n})\big)\bigg)(x_1,\dots,x_i)\notag\\
&\quad=\sum_{k=1}^n\frac{(-1)^i}{i!}\sum_{\eta\subseteq[x_1,\dots,x_i]}(-1)^{|\eta|}\langle\eta^{\otimes m},\varphi^{\otimes m}\rangle\langle\eta^{\otimes k},\so(n,k)\varphi^{\otimes n}\rangle\notag\\
&\quad=\sum_{k=1}^n\frac{(-1)^i}{i!}\sum_{\eta\subseteq[x_1,\dots,x_i]}(-1)^{|\eta|}\langle\eta^{\otimes m},\varphi^{\otimes m}\rangle\sum_{j=1}^n\langle (\eta)_j,\SO(k,j)\so(n,k)\varphi^{\otimes n}\rangle\notag\\
&\quad=\frac{(-1)^i}{i!}\sum_{\eta\subseteq[x_1,\dots,x_i]}(-1)^{|\eta|}\langle\eta^{\otimes m},\varphi^{\otimes m}\rangle\sum_{j=1}^n\Big\langle (\eta)_j,\sum_{k=1}^n\SO(k,j)\so(n,k)\varphi^{\otimes n}\Big\rangle\notag\\
&\quad=\frac{(-1)^i}{i!}\sum_{\eta\subseteq[x_1,\dots,x_i]}(-1)^{|\eta|}\langle\eta^{\otimes m},\varphi^{\otimes m}\rangle\langle(\eta)_n,\varphi^{\otimes n}\rangle\notag\\
&\quad=\frac{(-1)^i}{i!}\sum_{\eta\subseteq[x_1,\dots,x_i],\, |\eta|\ge n}(-1)^{|\eta|}\langle\eta^{\otimes m},\varphi^{\otimes m}\rangle\langle(\eta)_n,\varphi^{\otimes n}\rangle,\notag
\end{align}
which proves formula \eqref{xtsw5w}. If $i<n$, the right hand side of \eqref{xtsw5w} obviously vanishes. If $i>l$, the left hand side of \eqref{xtsw5w} vanishes. If $i=n$, the only $\eta\subseteq[x_1,\dots,x_i]$ that satisfies $|\eta|\ge i$ is $\eta=[x_1,\dots,x_i]$. Furthermore, by formula \eqref{reaq43q42}, we get 
$$(\delta_{x_1}+\dots+\delta_{x_i})_i=i!\,(\delta_{x_1}\odot\dots\odot\delta_{x_i}),$$
which implies that the right hand side of \eqref{xtsw5w} becomes \eqref{cxsesa53w}.
\end{proof}

\begin{remark}
In the case of a single-point space $X$, Proposition~\ref{dtrsews3a} implies 
$$\sum_{k=1}^nS(k+m,i)s(n,k)=\begin{cases}
0,&\text{if either $i<n$ or $i>n+m$},\\
i^m,&\text{if $i=n$},\\
\sum_{l=n}^i(-1)^{i+l}\frac{l^m}{(i-l)!\,(l-n)!},&\text{if }i=n+1,\dots,n+m.
\end{cases}$$
The case $i\le n$ is attributed in \cite[Section 12.2]{QG} to Frank Olson (1963). The case $i=n+1,\dots,n+m$ does not seem to have been considered before.
\end{remark}

\begin{remark} For a fixed $f^{(l)}$ as in Proposition~\ref{dtrsews3a}, denote
$$g^{(i)}:=\sum_{k=1}^n \SO(k+m,i)P_{m+k}\big(\mathbf 1^{(m)}\otimes \so(n,k)\big)f^{(l)}.$$
Then, by Theorem~\ref{rsdtresw45w} and Proposition~\ref{dtrsews3a}, we get
$$\langle (\omega)_n\odot\omega^{\otimes m},f^{(l)}\rangle=\sum_{i=n}^{n+m}\langle(\omega)_i,g^{(i)}\rangle.$$
\end{remark}

\begin{proposition} For any $n\in\mathbb N$, $i=0,1,\dots,n$, and $j=0,1,\dots,n-i$,
\begin{equation}\label{ceraq435q}
\binom{i+j}i\so(n,i+j)=\sum_{k=i}^{n-j}\binom nk P_{i+j}\big(\so(k,i)\otimes\so(n-k,j)\big).\end{equation}
In formula \eqref{ceraq435q}, the Stirling operators of the first kind can be replaced with the Stirling operators of the second kind. 
\end{proposition}

\begin{proof} To simplify the notation, we assume that $f^{(n)}\in\mathcal F^{(n)}(X)$ is of the form $f^{(n)}=\varphi^{\otimes n}$ with $\varphi\in\mathcal F(X)$. By \eqref{5w53w2} and \eqref{yrdqs}, we have,  for any $\omega,\sigma\in M(X)$,\begin{align}
\langle(\omega+\sigma)_n,f^{(n)}\rangle&=\sum_{k=0}^n\binom nk\langle(\omega)_k,\varphi^{\otimes k}\rangle\langle(\sigma)_{n-k},\varphi^{\otimes(n-k)}\rangle\notag\\
&=\sum_{k=0}^n\binom nk\sum_{i=0}^k\langle\omega^{\otimes i},\so(k,i)\varphi^{\otimes k}\rangle\sum_{j=0}^{n-k}\langle\sigma^{\otimes j},\so(n-k,j)\varphi^{\otimes(n-k)}\rangle\notag\\
&=\sum_{i=0}^n\sum_{j=0}^{n-i}\sum_{k=i}^{n-j}\binom nk \langle\omega^{\otimes i},\so(k,i)\varphi^{\otimes k}\rangle\langle\sigma^{\otimes j},\so(n-k,j)\varphi^{\otimes(n-k)}\rangle\notag\\
&=\sum_{i=0}^n\sum_{j=0}^{n-i}\Big\langle\omega^{\otimes i}\otimes\sigma^{\otimes j},\sum_{k=i}^{n-j}\binom nk \big(\so(k,i)\otimes\so(n-k,j)\big)f^{(n)}\Big\rangle.\label{cdtw5w}
\end{align}
On the other hand,
\begin{align}
\langle(\omega+\sigma)_n,f^{(n)}\rangle&=
\sum_{k=0}^n\langle(\omega+\sigma)^{\otimes k},\so(n,k)f^{(n)}\rangle\notag\\
&=\sum_{k=0}^n\sum_{i=0}^k\binom ki\langle\omega^{\otimes i}\odot\sigma^{\otimes(k-i)},\so(n,k)f^{(n)}\rangle\notag\\
&=\sum_{i=0}^n\sum_{j=0}^{n-i}\Big\langle\omega^{\otimes i}\odot\sigma^{\otimes j},\binom{i+j}i\so(n,i+j)f^{(n)}\Big\rangle.\label{sea3qa}
\end{align}
Let $u,v\in\mathbb F$, $\varkappa\in M(X)$ and set in formulas \eqref{cdtw5w}, \eqref{sea3qa} $\omega=u\varkappa$ and $\sigma=v\varkappa$. This gives
\begin{align*}
&\sum_{i=0}^n\sum_{j=0}^{n-i} u^iv^j \Big\langle\varkappa^{\otimes (i+j)},\sum_{k=i}^{n-j}\binom nk P_{i+j}\big[\so(k,i)\otimes\so(n-k,j)\big]f^{(n)}\Big\rangle\\
&\quad=\sum_{i=0}^n\sum_{j=0}^{n-i} u^iv^j \Big\langle\varkappa^{\otimes (i+j)},\binom{i+j}i\so(n,i+j)f^{(n)}\Big\rangle.
\end{align*}
This implies formula \eqref{ceraq435q}. The proof for the Stirling operators of the second kind is similar. One just uses formulas \eqref{5w53w22} and \eqref{fdr6ew456u4w} instead of \eqref{5w53w2} and \eqref{yrdqs}, respectively.  
\end{proof}

Recall that, in Subsection \ref{fdsters5w5}, we defined $\mathcal F(\ddot\Gamma_0(X))$ as the algebraic direct sum of the $\mathcal F^{(n)}(X)$ spaces. For each $x\in X$, we define a linear operator $\mathcal D_x\in\mathcal L(\mathcal F(\ddot\Gamma_0(X)))$  that maps each $\mathcal F^{(n)}(X)$ into $\mathcal F^{(n-1)}(X)$ so that, for each $f^{(n)}\in\mathcal F^{(n)}(X)$, we have 
$(\mathcal D_xf^{(n)})(\cdot):=nf^{(n)}(x,\cdot)$. Note that, by \eqref{ydsrqqsqwdx}, $\partial_x\langle\omega^{\otimes n},f^{(n)}\rangle=\langle\omega^{\otimes(n-1)},\mathcal D_xf^{(n)}\rangle$.

\begin{proposition}\label{drtes5w3}
For each $n,i\in\mathbb N$, $i<n$, and $f^{(n)}\in\mathcal F^{(n)}(X)$, 
\begin{align}\label{ts6uw65uxz}
\sum_{k=1}^{n-i}\frac1{k!}\, (-\mathcal D_x)^k\,\so(n,i+k)f^{(n)}&=\sum_{k=1}^{n-i}\so(n-k,i)(-\mathcal D_x)^kf^{(n)},\\
\sum_{k=1}^{n-i}(-\mathcal D_x)^k\,\SO(n,i+k)f^{(n)}&=\sum_{k=1}^{n-i}\frac1{k!}\,\SO(n-k,i)(-\mathcal D_x)^kf^{(n)}.\label{utfrtrz}
\end{align}
\end{proposition}

\begin{proof} To simplify the notation, we assume that $f^{(n)}\in\mathcal F^{(n)}(X)$ is of the form $f^{(n)}=\varphi^{\otimes n}$ with $\varphi\in\mathcal F(X)$. By \eqref{xesa332a}, we get $(-\delta_x)_n=(-1)^nn!\,\delta_x^{\otimes n}$. Hence, by \eqref{5w53w2}, for each $\omega\in M(X)$,
$$(\omega-\delta_x)_n=\sum_{k=0}^n (-1)^{n-k}\,(n)_{n-k}\, (\omega)_k\odot\delta_x^{\otimes(n-k)},$$
which implies
\begin{align}
\langle (\omega-\delta_x)_n,f^{(n)}\rangle&=
\sum_{k=0}^n(-1)^{n-k}(n)_{n-k}\,\langle(\omega)_k,\varphi^{\otimes k}\rangle\,\varphi^{n-k}(x)\notag\\
&=\sum_{k=0}^n(-1)^{n-k}(n)_{n-k}\,\sum_{i=0}^k\langle\omega^{\otimes i},\so(k,i)\varphi^{\otimes k}\rangle\,\varphi^{n-k}(x)\notag\\
&=\sum_{k=0}^n\sum_{i=0}^k\langle\omega^{\otimes i},\so(k,i)(-\mathcal D_x)^{n-k}f^{(n)}\rangle\notag\\
&=\sum_{i=0}^n\Big\langle\omega^{\otimes i},\sum_{k=i}^n \so(k,i)(-\mathcal D_x)^{n-k}f^{(n)}\Big\rangle\notag\\
&=\sum_{i=0}^n\Big\langle\omega^{\otimes i},\so(n,i)f^{(n)}+\sum_{k=1}^{n-i}\so(n-k,i)(-\mathcal D_x)^{k}f^{(n)}\Big\rangle.
\label{w54qw5}\end{align}
On the other hand,
\begin{align}
\langle (\omega-\delta_x)_n,f^{(n)}\rangle&=\sum_{k=0}^n\langle(\omega-\delta_x)^{\otimes k},\so(n,k)f^{(n)}\rangle\notag\\
&=\sum_{k=0}^n\sum_{i=0}^k\binom ki\langle \omega^{\otimes i}\odot (-\delta_x)^{\otimes(k-i)},\so(n,k)f^{(n)}\rangle\notag\\
&=\sum_{k=0}^n\sum_{i=0}^k\frac1{(k-i)!}\,\langle\omega^{\otimes i},(-\mathcal D_x)^{k-i}\so(n,k)f^{(n)}\rangle\notag\\
&=\sum_{i=0}^n\Big\langle\omega^{\otimes i},\so(n,i)f^{(n)}+\sum_{k=1}^{n-i} \frac1{k!}\, (-\mathcal D_x)^{k}\so(n,i+k)f^{(n)}\Big\rangle.\label{rsw553q5}
\end{align}
Now formula \eqref{ts6uw65uxz} follows from \eqref{w54qw5} and \eqref{rsw553q5}. The proof of \eqref{utfrtrz} is similar, one just starts with the polynomial $\langle(\omega-\delta_x)^{\otimes n},f^{(n)}\rangle$.
\end{proof}

\begin{remark}
In the case of a single-point space $X$, Proposition~\ref{drtes5w3} gives the following identities:
\begin{align}
\sum_{k=1}^{n-i}(-1)^k\binom{i+k}{k}s(n,i+k)&=\sum_{k=1}^{n-i}(-1)^k (n)_k\, s(n-k,i),\label{ds5w53wq}\\
\sum_{k=1}^{n-i}(-1)^k(i+k)_k\,S(n,i+k)&=\sum_{k=1}^{n-i}(-1)^k\binom nk S(n-k,i).\notag
\end{align}
Formula \eqref{ds5w53wq} is the equality (12.15) in \cite{QG}. 
\end{remark}

\begin{remark}
In equalities \eqref{ceraq435q} and \eqref{ts6uw65uxz},  the Stirling operators of the first kind can be replaced with the Lah operators.
\end{remark}

\section{Poisson functional and Stirling operators}\label{ufxdqyr}

Let $\omega\in M(X)$ be a positive Radon measure and assume that $\omega$ is non-atomic, i.e., $\omega(\{x\})=0$ for all $x\in X$. Then one can define a Poisson point process on $X$ with intensity measure $\omega$, which is a probability measure on the configuration space $\Gamma(X)$ equipped with the cylinder $\sigma$-algebra, see e.g.\ \cite{Kingman}. Denote by $\mathbb E_\omega$ the expectation with respect to this point process. The restriction of each polynomial $p\in\mathcal P(M(X))$ to $\Gamma(X)$ is a random variable on $\Gamma(X)$. Let $p\in \mathcal P(M(X))$ be of the form $p(\omega)=\sum_{k=0}^n\langle\omega^{\otimes k},f^{(k)}\rangle$, and choose a set $\Lambda\in\mathcal B_0(X)$ such that, for each $k=1,\dots,n$, the function $f^{(k)}$ vanishes outside $\Lambda^k$. Then, it follows from the definition of the Poisson point process that
\begin{equation}\label{ufyde}
\mathbb E_\omega(p)=e^{-\omega(\Lambda)}\sum_{n=0}^\infty\frac1{n!}\int_{\Lambda^n}p([x_1,\dots,x_n])\,\omega^{\otimes n}(dx_1\dotsm dx_n).\end{equation}
(In the sum on the right-hand side of formula \eqref{ufyde}, the term corresponding to $n=0$ is meant to be $p(0)$.) Note that the value of the right hand side of formula \eqref{ufyde}  does not depend on the choice of the set $\Lambda$. Furthermore, since the measure $\omega$ is non-atomic, for $\omega^{\otimes n}$-a.a.\ $(x_1,\dots,x_n)\in \Lambda^n$, $[x_1,\dots,x_n]\in \Gamma_0(X)$. 

From now on, we assume that $\omega\in M(X)$ is arbitrary. We define a  (linear) {\it Poisson functional\/  $\mathbb E_\omega:\mathcal P(M(X))\to\R$ with intensity measure $\omega$} by formula \eqref{ufyde}, provided the set $\Lambda$ satisfies the above assumption that each $f^{(k)}$ vanishes outside $\Lambda^k$. In Appendix,  we prove that the value of the right-hand side of formula~\eqref{ufyde} still does not depend on the choice of such a set $\Lambda$.  
Furthermore, we discuss in Appendix several properties of the Poisson functional which generalize the corresponding facts about the Poisson point process.

\begin{theorem} \label{e6ew63} Let $f^{(n)}\in\mathcal F^{(n)}(X)$ and let $p(\omega)=\langle\omega^{\otimes n},f^{(n)}\rangle$,  $\omega\in M(X)$. Then
$$\mathbb E_\omega(p)=\sum_{k=1}^n\langle\omega^{\otimes k},\SO(n,k)f^{(n)}\rangle.$$
\end{theorem}

\begin{proof}
We need equivalently to prove that, for each falling factorial $p\in\mathcal P(M(X))$ of the form $p(\omega)=\langle(\omega)_k,g^{(k)}\rangle$ with $g^{(k)}\in\mathcal F^{(k)}(X)$, we have $\mathbb E_\omega(p)=\langle\omega^{\otimes k},g^{(k)}\rangle$. We prove the statement by induction on $k$. For $k=1$, the result follows immediately from Proposition~\ref{g6ew4}.

Assume that the statement holds for $k$ and let us prove it for $k+1$. Choose $\Lambda\in\mathcal B_0(X)$ such that $g^{(k+1)}$ vanishes outside of $\Lambda^{k+1}$. Let $\eta\in\ddot\Gamma_0(\Lambda)$, $\eta=[x_1,\dots,x_m]$, $m\in\mathbb N$. By \eqref{xesa332a},
\begin{align}
\langle(\eta)_{k+1},g^{(k+1)}\rangle&=\sum_{i_1\in\{1,\dots,m\}}\,\sum_{i_2\in\{1,\dots,m\}\setminus\{i_1\}}\dotsm \sum_{i_{k+1}\in\{1,\dots,m\}\setminus\{i_1,\dots,i_{k}\}}g^{(k+1)}(x_{i_1},\dots,x_{i_{k+1}})\notag\\
&=\int_\Lambda\eta(dx)\int_{\Lambda^k}(\eta-\delta_x)_k(dx_1\dotsm dx_k)g^{(k+1)}(x,x_1,\dots,x_k).\label{crse3aq3}
\end{align}
By \eqref{crse3aq3}, Proposition \ref{g6ew4}, and the induction assumption, we have, for $p(\omega)=\langle(\omega)_{k+1},g^{(k+1)}\rangle$,
\begin{align*}
\mathbb E_\omega(p)&=\int_{\ddot\Gamma_0(\Lambda)}\mathbb P_\omega^\Lambda(d\eta)\int_\Lambda \eta(dx)\int_{\Lambda^k}(\eta-\delta_x)_{k}(dx_1\dotsm dx_k)g^{(k+1)}(x,x_1,\dots,x_k)\\
&=\int_{\ddot\Gamma_0(\Lambda)}\mathbb P_\omega^\Lambda(d\eta)\int_\Lambda \omega(dx)\int_{\Lambda^k}(\eta)_k(dx_1\dotsm dx_k)g^{(k+1)}(x,x_1,\dots,x_k)\\
&=\int_\Lambda \omega(dx)\int_{\ddot\Gamma_0(\Lambda)}\mathbb P_\omega^\Lambda(d\eta)\int_{\Lambda^k}(\eta)_k(dx_1\dotsm dx_k)g^{(k+1)}(x,x_1,\dots,x_k)\\
&=\langle\omega^{\otimes(k+1)},g^{(k+1)}\rangle.\qedhere
\end{align*}
\end{proof}

\begin{remark} Consider a linear operator $U\in\mathcal L(\mathcal P(M(X)))$ defined by $(Up)(\omega):=\mathbb E_\omega(p)$ for $p\in\mathcal P(M(X))$. It follows from the proof of Theorem~\ref{e6ew63} that $U$ maps each falling factorial $\langle(\omega)_k,g^{(k)}\rangle$ into the monomial $\langle\omega^{\otimes k},g^{(k)}\rangle$. Note that  both the falling factorials and the monomials have the binomial property. Hence, by analogy with umbral calculus in dimension one (e.g.\ \cite{Roman}), we can think of $U$ as an {\it umbral operator}.
\end{remark}

\begin{corollary}\label{f6we5w}
Let $p\in \mathcal P(M(X))$ be of the form $p(\omega)=\sum_{k=0}^n\langle\omega^{\otimes k},f^{(k)}\rangle$, and choose a set $\Lambda\in\mathcal B_0(X)$ such that, for each $k=1,\dots,n$, the function $f^{(k)}$ vanishes outside $\Lambda^k$. Then 
$$\mathbb E_\omega(p)=\sum_{i=0}^n\frac1{i!}\int_{\Lambda^i}p([x_1,\dots,x_{i}])\,\omega^{\otimes i}(dx_1\dotsm dx_{i}) \sum_{k=0}^{n-i} \frac{\big(-\omega(\Lambda)\big)^{k}}{k!}.$$
\end{corollary}

\begin{proof} By Theorem \ref{rsdtresw45w} and (the proof of) Theorem~\ref{e6ew63}, we have
\begin{align*}
\mathbb E_\omega(p)&=\sum_{k=0}^n\frac{(-1)^k}{k!}\int_{\Lambda^k}\sum_{i=0}^k(-1)^i\sum_{\{l_1,\dots,l_i\}\subset\{1,\dots,k\}}p([x_{l_1},\dots,x_{l_i}])\,\omega^{\otimes k}(dx_1\dotsm dx_k)\\
&=\sum_{k=0}^n\sum_{i=0}^k\frac{(-1)^{k-i}}{k!}\binom ki
\int_{\Lambda^k}p([x_1,\dots,x_i])\,\omega^{\otimes k}(dx_1\dotsm dx_k)\\
&=\sum_{k=0}^n\sum_{i=0}^k\frac{(-1)^{k-i}}{i!(k-i)!}\,\omega(\Lambda)^{k-i}
\int_{\Lambda^i}p([x_1,\dots,x_i])\,\omega^{\otimes i}(dx_1\dotsm dx_i),\end{align*}
which implies the statement.
\end{proof}

\begin{remark}
In the case of a single-point space $X$, Corollary \ref{f6we5w} states that, if $p$ is a polynomial on $\mathbb F$ of degree $n$, then for each $z\in\mathbb F$,
$$\sum_{k=0}^\infty\frac{z^k}{k!} \, p(k)=e^z \sum_{i=0}^n \frac{p(i)}{i!}\sum_{k=0}^{n-i}\frac{(-1)^k z^{k+i}}{k!},
$$
which is Theorem~9.2 in \cite{QG}.
\end{remark}

\section{Infinite dimensional Gr\"unert's and Katriel's formulas}\label{gwfdytwyc}

\subsection{Infinite dimensional Gr\"unert's formula}

For each $\xi\in\mathcal F(X)$, we consider the following linear operator on $\mathcal P(M(X))$:
$$\langle\omega,\xi\partial\rangle=\int_X\omega(dx)\xi(x)\partial_x,$$
 where $\partial_x$ is defined by 
\eqref{qdftyfe} and $\omega\in M(X)$ is the variable of the polynomial this operator is applied to. More precisely, by \eqref{ydsrqqsqwdx} and \eqref{cresa5}, for each $f^{(n)}\in\mathcal F^{(n)}(X)$, we get
\begin{equation}\label{fde5w44}
\langle\omega,\xi\partial\rangle\langle\omega^{\otimes n},f^{(n)}\rangle=n\int_X\omega(dx)\xi(x)\langle\omega^{\otimes(n-1)},f^{(n)}(x,\cdot)\rangle=\langle\omega^{\otimes n},N(\xi)f^{(n)}\rangle.
 \end{equation}
 We similarly define, for $f^{(n)}\in\mathcal F^{(n)}(X)$, the operator
 $$\langle\omega^{\otimes n},f^{(n)}\partial^{\otimes n}\rangle=\int_{X^n}\omega^{\otimes n}(dx_1\dotsm dx_n)f^{(n)}(x_1,\dots,x_n)\partial_{x_1}\dotsm\partial_{x_n}.$$

\begin{theorem}[Infinite dimensional Gr\"unert's formula]\label{dwq3q32q}
 For any $\xi_1,\dots,\xi_n\in\mathcal F(X)$,
\begin{equation}
\langle\omega,\xi_1\partial\rangle\dotsm \langle\omega,\xi_n\partial\rangle=\sum_{k=1}^n\langle\omega^{\otimes k},\big(\SO(n,k)(\xi_1\odot\dots\odot\xi_n)\big)\partial^{\otimes k}\rangle.\label{cdrts6ew6}
\end{equation}
\end{theorem}

\begin{proof} We start with the following

\begin{lemma}\label{yd6effd6} For any $\xi\in\mathcal F(X)$ and $f^{(n)}\in\mathcal F^{(n)}(X)$,
\begin{equation}\notag
\langle\omega,\xi\partial\rangle
\langle\omega^{\otimes n},f^{(n)}\partial^{\otimes n}\rangle=\langle\omega^{\otimes n},(N(\xi)f^{(n)})\partial^{\otimes n}\rangle+\langle\omega^{\otimes(n+1)},(\xi\odot f^{(n)})\partial^{\otimes(n+1)}\rangle.
\end{equation}
\end{lemma}

\begin{proof} Let $g^{(m)}\in\mathcal F^{(m)}(X)$, and to simplify the notation we assume that $g^{(m)}=\varphi^{\otimes m}$ with $\varphi\in\mathcal F(X)$. Then
\begin{align}
&\langle\omega,\xi\partial\rangle
\langle\omega^{\otimes n},f^{(n)}\partial^{\otimes n}\rangle\langle\omega^{\otimes m},g^{(m)}\rangle\notag\\
&\quad=\langle\omega,\xi\partial\rangle(m)_n\langle\omega^{\otimes n},f^{(n)}\varphi^{\otimes n}\rangle\langle\omega^{\otimes(m-n)},\varphi^{\otimes(m-n)}\rangle\notag\\
&\quad=(m)_n \langle\omega^{\otimes n},N(\xi)(f^{(n)}\varphi^{\otimes n})\rangle\langle\omega^{\otimes(m-n)},\varphi^{\otimes(m-n)}\rangle\notag\\
&\qquad+(m)_{n+1}\langle\omega^{\otimes n},f^{(n)}\varphi^{\otimes n}\rangle\langle\omega,\xi\varphi\rangle\langle\omega^{\otimes(m-n-1)},\varphi^{\otimes(m-n-1)}\rangle\notag\\
&\quad=\langle\omega^{\otimes n},(N(\xi)f^{(n)})\partial^{\otimes n}\rangle\langle\omega^{\otimes m},\varphi^{\otimes m}\rangle\notag\\
&\qquad+(m)_{n+1}\langle\omega^{\otimes(n+1)},(f^{(n)}\odot\xi)\varphi^{\otimes(n+1)}\rangle\langle\omega^{\otimes(m-n-1)},\varphi^{\otimes(m-n-1)}\rangle\notag\\
&\quad=\big[\langle\omega^{\otimes n},(N(\xi)f^{(n)})\partial^{\otimes n}\rangle+\langle\omega^{\otimes(n+1)},(\xi\odot f^{(n)})\partial^{\otimes(n+1)}\rangle\big]\langle\omega^{\otimes m},\varphi^{\otimes m}\rangle.\notag\qedhere
\end{align}
\end{proof}

We now prove \eqref{cdrts6ew6} by induction on $n$. For $n=1$, \eqref{cdrts6ew6}  is just trivial. Assume that \eqref{cdrts6ew6} holds for $n$. Noting that $(\langle\omega,\xi\partial\rangle)_{\xi\in\mathcal F(X)}$ is a family of commuting operators, we get by the induction assumption, \eqref{rses3waw2}, \eqref{crtw5w}, and Lemma~\ref{yd6effd6}, 
\begin{align}
&\langle\omega,\xi_1\partial\rangle\dotsm \langle\omega,\xi_{n+1}\partial\rangle=\sum_{k=1}^n\langle\omega,\xi_{n+1}\partial\rangle\langle\omega^{\otimes k},(\SO(n,k)(\xi_1\odot\dots\odot\xi_n))\partial^{\otimes k}\rangle\notag\\
&\quad=\langle\omega^{\otimes (n+1)},(\SO(n+1,n+1)(\xi_1\odot\dots\odot\xi_{n+1}))\partial^{\otimes {n+1}}\rangle\notag\\
&\qquad+\sum_{k=2}^n\big[\langle\omega^{\otimes k},(\xi_{n+1}\odot(\SO(n,k-1)(\xi_1\odot\dots\odot\xi_n)))\partial^{\otimes k}\rangle\notag\\
&\qquad+\langle\omega^{\otimes k},(N(\xi_{n+1})\SO(n,k)(\xi_1\odot\dots\odot\xi_n))\partial^{\otimes k}\rangle\big]+\langle\omega,\xi_{n+1}(\SO(n,1)(\xi_1\odot\dots\odot\xi_n))\rangle\notag\\
&\quad=\sum_{k=1}^{n+1}\langle\omega^{\otimes k},(\SO(n+1,k)(\xi_1\odot\dots\odot\xi_{n+1}))\partial^{\otimes k}\rangle.\notag\qedhere
\end{align}
\end{proof}

\begin{corollary}\label{vctrse45w}
For any $\xi,\varphi\in\mathcal F(X)$ and $n\in\mathbb N$, we have
\begin{align}
&\langle\omega,\xi\partial\rangle^n\frac1{1-\langle\omega,\varphi\rangle}=\sum_{k=1}^n\langle\omega^{\otimes k},(\SO(n,k)\xi^{\otimes n})\varphi^{\otimes k}\rangle\frac{k!}{(1-\langle\omega,\varphi\rangle)^{k+1}},\label{f6ew5w}\\
&\langle\omega,\xi\partial\rangle^n e^{\langle\omega,\varphi\rangle}=\sum_{k=1}^n\langle\omega^{\otimes k},(\SO(n,k)\xi^{\otimes n})\varphi^{\otimes k}\rangle e^{\langle\omega,\varphi\rangle}.\label{drtwe54w}
\end{align}
\end{corollary}

\begin{proof}
Both formulas \eqref{f6ew5w} and \eqref{drtwe54w} follow directly from Theorem \ref{dwq3q32q} by differentiation of the functions $\frac1{1-\langle\omega,\varphi\rangle}$ and $e^{\langle\omega,\varphi\rangle}$, respectively.
\end{proof}

\begin{remark}
Formula \eqref{f6ew5w} is an infinite dimensional extension of formula (9.47) in \cite{QG}.
\end{remark}

\subsection{Wick ordering and infinite dimensional Katriel's formula}

 Let $V$ be a vector space. For linear operators $A,B\in\mathcal L(V)$,   we denote $[A,B]:=AB-BA$, called the {\it commutator of $A$ and $B$}. 

Let us fix a reference measure $\sigma\in M(X)$.  Let us consider linear operators $a^+(\xi),a^-(\xi)\in\mathcal L(V)$ that linearly depend  on $\xi\in\mathcal F(X)$ and satisfy the {\it canonical commutation relations\footnote{We drop the standard assumption that the measure $\sigma$ appearing in the canonical commutation relations is positive. (Hence, we neither assume that $V$ is a Hilbert space, nor that  $a^-(\xi)$ is the adjoint of $a^+(\xi)$.)} (CCR)}:
\begin{gather}
[a^+(\varphi),a^+(\xi)]=[a^-(\varphi),a^-(\xi)]=0,\notag\\
[a^-(\varphi),a^+(\xi)]=\int_X\varphi(x) \xi(x)\sigma(dx),\quad \varphi,\xi\in\mathcal F(X).\label{xra453}
\end{gather}
The operators $a^+(\xi)$ and $a^-(\xi)$ are called {\it creation} and {\it annihilation operators}, respectively. 

 Our aim now is to introduce in $V$ the corresponding operators of particle density and Wick product of these operators. We will initially do this heuristically. 

We define in $V$ {\it creation operators $a^+(x)$} and {\it annihilation operators 
$a^-(x)$ at points $x\in X$} that satisfy
$$a^+(\xi)=\int_X \xi(x)a^+(x)\sigma(dx),\quad a^-(\xi)=\int_X\xi(x)a^-(x)\sigma(dx)\quad\text{for all }\xi\in \mathcal F(X).$$
In terms of these operators, the CCR \eqref{xra453} become
\begin{equation}\label{r4wqa4y}
[a^+(x),a^+(y)]=[a^-(x),a^-(y)]=0,\quad [a^-(x),a^+(y)]=\delta(x,y),
\end{equation}
where the distribution $\delta(x,y)$ satisfies 
$$\int_{X^2}\varphi(x)\xi(y) \delta(x,y)\sigma(dx)\sigma(dy)=\int_X\varphi(x)\xi(x)\sigma(dx),\quad\xi,\varphi\in\mathcal F(X).$$

We define the {\it particle density} $\rho(x):=a^+(x)a^-(x)$ for $x\in X$, and the corresponding {\it operators of particle density} $\rho(\xi):=\int_X\xi(x)\rho(x)\sigma(dx)$ for $\xi\in\mathcal F(X)$.
The CCR \eqref{r4wqa4y} then imply the commutation relations
\begin{equation}\label{esa4q4z}
[\rho(\varphi),\rho(\xi)]=0,\quad [\rho(\varphi),a^+(\xi)]=a^+(\varphi\xi),\quad [a^-(\xi),\rho(\varphi)]=a^-(\varphi\xi).\end{equation}
for all $\varphi,\xi\in\mathcal F(X)$.

Let $\sharp_1,\dots,\sharp_n\in\{+,-\}$. We define the {\it Wick product} ${:}\,a^{\sharp_1}(x_1)\dotsm a^{\sharp_n}(x_n){:}$ as the product of the operators $a^{\sharp_1}(x_1),\dots, a^{\sharp_n}(x_n)$ that is Wick ordered, i.e.,  all the creation operators are to the left of all the annihilation operators.  In particular, 
\begin{equation}\label{cd55uw}
{:}\,\rho(x_1)\dotsm \rho(x_n){:}=a^+(x_1)\dotsm a^+(x_n)a^-(x_n)\dotsm a^-(x_1).\end{equation}

\begin{lemma}\label{erws5w3}
 The CCR \eqref{r4wqa4y}  imply the following   recurrence formulas:
\begin{gather}
{:}\rho(x){:}=\rho(x),\notag\\
{:}\rho(x_1)\dotsm \rho(x_n){:}=\rho(x_1)\,{:}\rho(x_2)\dotsm \rho(x_n){:}
-\sum_{i=2}^n\delta(x_1,x_i)\,{:}\rho(x_2)\dotsm \rho(x_n){:}\, ,\quad n\ge2.\notag
\end{gather}
\end{lemma}

\begin{proof} This formal result is known, see e.g.\ \cite[Section~2]{MeSh}. For the reader's convenience, we present the explicit calculations:
\begin{align*}
&\rho(x_1)\,{:}\rho(x_2)\dotsm \rho(x_n){:}=a^+(x_1)a^-(x_1)a^+(x_2)\dotsm a^+(x_n)a^-(x_n)\dotsm a^-(x_2)\\
&=\delta(x_1,x_2)a^+(x_1)a^+(x_3)\dotsm a^+(x_n)a^-(x_n)\dotsm a^-(x_2)\\
&\quad+a^+(x_1)a^+(x_2)a^-(x_1)a^+(x_3)\dotsm a^+(x_n)a^-(x_n)\dotsm a^-(x_2)\\
&=\delta(x_1,x_2){:}\rho(x_2)\dotsm \rho(x_n){:}+\delta(x_1,x_3)a^+(x_1)a^+(x_2)a^+(x_4)\dotsm a^+(x_n)a^-(x_n)\dotsm a^-(x_2)\\
&\quad+a^+(x_1)a^+(x_2)a^+(x_3)a^-(x_1)a^+(x_4)\dotsm a^+(x_n)a^-(x_n)\dotsm a^-(x_2)\\
&=\sum_{i=2}^3\delta(x_1,x_i){:}\rho(x_2)\dotsm \rho(x_n){:}\\
&\quad+a^+(x_1)a^+(x_2)a^+(x_3)a^-(x_1)a^+(x_4)\dotsm a^+(x_n)a^-(x_n)\dotsm a^-(x_2)\\
&=\sum_{i=2}^n\delta(x_1,x_i){:}\rho(x_2)\dotsm \rho(x_n){:}+a^+(x_1)\dotsm a^+(x_n)a^-(x_1)a^-(x_n)\dotsm a^-(x_2)\\
&=\sum_{i=2}^n\delta(x_1,x_i){:}\rho(x_2)\dotsm \rho(x_n){:}+
{:}\rho(x_1)\dotsm \rho(x_n){:}\,.\qedhere
\end{align*}
\end{proof}

  We are now in position to rigorously treat the  Wick product of operators of particle density. Assume that we are given a family of creation operators $a^+(\xi)$ and annihilation operators $a^-(\xi)$ acting in a vector space $V$ and satisfying the CCR \eqref{xra453}. Further  assume that the corresponding particle density $\rho(\xi)$ ($\xi\in\mathcal F(X)$) is well defined as a family of linear operators in $V$ that linearly depend on $\xi$. Hence, the commutation relations \eqref{esa4q4z} hold.

 Inspired by Lemma~\ref{erws5w3},  we  define, for $\xi\in\mathcal F(X)$, $\int_X \xi(x){:}\rho(x){:}\,\sigma(dx):=\rho(\xi)$ and for $\xi_1,\dots,\xi_n\in\mathcal F(X)$ ($n\ge2$),
\begin{align}
&\int_{X^n}\xi_1(x_1)\dotsm\xi_n(x_n)\,{:}\rho(x_1)\dotsm \rho(x_n){:}\,\sigma^{\otimes n}(dx_1\dotsm dx_n)\notag\\
&:=\rho(\xi_1)\int_{X^{n-1}}\xi_2(x_2)\dotsm\xi_n(x_n)\,{:}\rho(x_2)\dotsm \rho(x_n){:}\,\sigma^{\otimes (n-1)}(dx_2\dotsm dx_n)
\notag\\
&
-\sum_{i=2}^n\int_{X^n}\xi_2(x_2)\dotsm (\xi_1\xi_i)(x_i)\dotsm
\xi_n(x_n)\,{:}\rho(x_2)\dotsm \rho(x_n){:}\,\sigma^{\otimes (n-1)}(dx_2\dotsm dx_{n}).
\label{tesw5yw}
\end{align}
Extending this by linearity, we obtain a linear operator
$$\int_{X^n}f^{(n)}(x_1,\dots,x_n)\,{:}\rho(x_1)\dotsm \rho(x_n){:}\,\sigma^{\otimes n}(dx_1\dotsm dx_n)\in\mathcal L(V)$$
 for each $f^{(n)}:X^n\to\mathbb F$ that is a linear combination of functions $\xi_1\otimes\dots\otimes \xi_n$ with  $\xi_1,\dots,\xi_n\in\mathcal F(X)$.

\begin{theorem}[Infinite dimensional Katriel's formula] \label{taw4aq234q}
Let  $\rho(\xi)$ ($\xi\in\mathcal F(X)$) be the operators of particle density for a family of creation  operators $a^+(\xi)$ and annihilation operators $a^-(\xi)$ satisfying the CCR  \eqref{xra453}. Then, for each $n\in\mathbb N$ and $\xi_1,\dots,\xi_n\in\mathcal F(X)$,
\begin{align}
&\rho(\xi_1)\dotsm\rho(\xi_n)\notag\\
&=\sum_{k=1}^n\int_{X^k}(\SO(n,k)(\xi_1\odot\dotsm\odot\xi_n))(x_1,\dots,x_k)\,{:}\rho(x_1)\dotsm\rho(x_k){:}\,\sigma^{\otimes k}(dx_1\dotsm dx_k).\label{ds5w5u4} 
\end{align}
\end{theorem}

\begin{remark}
In the case where $X=\mathbb R^d$ and $\sigma$ is the Lebesgue measure, a formula equivalent to \eqref{ds5w5u4} was discussed, at a formal level,  by 
Menikoff and Sharp in \cite[Section~2]{MeSh}.
\end{remark}

\begin{proof}[Proof of Theorem \ref{taw4aq234q}] As easily follows from \eqref{tesw5yw}, for any $\xi_1,\dots,\xi_n\in\mathcal F(X)$ and a permutation $\pi\in \mathfrak S_n$,
\begin{align*}
&\int_{X^n}\xi_1(x_1)\dotsm\xi_n(x_n)\,{:}\rho(x_1)\dotsm \rho(x_n){:}\,\sigma^{\otimes n}(dx_1\dotsm dx_n)\\
&\quad=\int_{X^n}\xi_{\pi(1)}(x_1)\dotsm\xi_{\pi(n)}(x_n)\,{:}\rho(x_1)\dotsm \rho(x_n){:}\,\sigma^{\otimes n}(dx_1\dotsm dx_n),
\end{align*}
which implies
\begin{align}
&\int_{X^n}\xi_1(x_1)\dotsm\xi_n(x_n)\,{:}\rho(x_1)\dotsm \rho(x_n){:}\,\sigma^{\otimes n}(dx_1\dotsm dx_n)\notag\\
&\quad=\int_{X^n}(\xi_{1}\odot\dotsm\odot\xi_n)(x_1,\dots,x_n)\,{:}\rho(x_1)\dotsm \rho(x_n){:}\,\sigma^{\otimes n}(dx_1\dotsm dx_n).\label{cdtrs5w}
\end{align}
By \eqref{cresa5}, \eqref{tesw5yw}, and \eqref{cdtrs5w}, for any $\varphi,\xi_1,\dots,\xi_n\in\mathcal F(X)$,
\begin{align}
&\rho(\varphi)\int_{X^n}(\xi_{1}\odot\dotsm\odot\xi_n)(x_1,\dots,x_n)\,{:}\rho(x_1)\dotsm \rho(x_n){:}\,\sigma^{\otimes n}(dx_1\dotsm dx_n)\notag\\
&\quad=\int_{X^{n+1}}(\varphi\odot\xi_{1}\odot\dotsm\odot\xi_n)(x_1,\dots,x_{n+1})\,{:}\rho(x_1)\dotsm \rho(x_{n+1}){:}\,\sigma^{\otimes (n+1)}(dx_1\dotsm dx_{n+1})\notag\\
&\qquad+\int_{X^n}(N(\varphi)(\xi_{1}\odot\dotsm\odot\xi_n))(x_1,\dots,x_n)\,{:}\rho(x_1)\dotsm \rho(x_n){:}\,\sigma^{\otimes n}(dx_1\dotsm dx_n).\notag
\end{align}
Now to prove formula \eqref{ds5w5u4} we employ the same arguments as those used to derive formula \eqref{cdrts6ew6} from Lemma \ref{yd6effd6}.
\end{proof}

\begin{remark}\label{tesw5}
We note that Theorem~\ref{dwq3q32q} is actually a special case of Theorem~\ref{taw4aq234q}. Indeed, let $V=\mathcal P(M(X))$, and for each $\xi\in\mathcal F(X)$, we define $a^+(\xi)=\langle\omega,\xi\rangle$, the operator of multiplication by $\langle\omega,\xi\rangle$, and $a^-(\xi)=\langle\sigma,\xi\partial\rangle$. These operators satisfy the CCR \eqref{xra453}. In this case, $a^-(x)=\partial_x$ is a well-defined operator on $V$, while $a^+(x)$ is an operator-valued distribution---the operator of multiplication by the (generally speaking distribution) $\frac{d\omega}{d\sigma}(x)$. The corresponding particle density is
$\rho(x)=\frac{d\omega}{d\sigma}(x)\partial_x$, which yields, for each $\xi\in\mathcal F(X)$,
$$\rho(\xi)=\int_X\xi(x)\frac{d\omega}{d\sigma}(x)\partial_x\,\sigma(dx)=\int_X\omega(dx)\xi(x)\partial_x=\langle\omega,\xi\partial\rangle.$$
 As easily seen, the operators $a^+(\xi)$, $a^-(\xi)$, $\rho(\xi)$ indeed satisfy \eqref{esa4q4z}. Furthermore, using \eqref{tesw5yw}, we conclude that,  for any $\xi_1,\dots,\xi_n\in\mathcal F(X)$,
\begin{align}
&\int_{X^n}\xi_1(x_1)\dotsm\xi_n(x_n)\,{:}\rho(x_1)\dotsm \rho(x_n){:}\,\sigma^{\otimes n}(dx_1\dotsm dx_n)\notag\\
&\quad=\langle\omega^{\otimes n},(\xi_1\otimes\dots\otimes \xi_n)\partial^{\otimes n}\rangle= \langle\omega^{\otimes n},(\xi_1\odot\dots\odot \xi_n)\partial^{\otimes n}\rangle.\notag
\end{align}
Hence, formula \eqref{cdrts6ew6} is  a consequence of the commutation relations between the operators of multiplication by $\langle\omega,\xi\rangle$ ($\xi\in\mathcal F(X)$) and the differentiation operators $\partial_x$ ($x\in X$). Note also that, in this case, the choice of the reference measure $\sigma\in M(X)$ was irrelevant. 
\end{remark}

\begin{remark} Formula \eqref{ds5w5u4} can be inverted:
\begin{align*}
&\int_{X^n}\xi_1(x_1)\dotsm\xi_n(x_n)\,{:}\rho(x_1)\dotsm\rho(x_n){:}\,\sigma^{\otimes n}(dx_1\dots,dx_n)\\
&\quad=\sum_{k=1}^n\int_{X^k}(\so(n,k)(\xi_1\odot\dotsm\odot\xi_n))(x_1,\dots,x_k)\,\rho(x_1)\dotsm\rho(x_k)\,\sigma^{\otimes k}(dx_1\dotsm dx_k).
\end{align*}
This follows immediately from  \eqref{ds5w5u4} and Proposition \ref{fye6e}.
\end{remark}

\subsection{Quantum Poisson process}

Let the conditions of Theorem \ref{taw4aq234q} be satisfied. 
 Let $\mathbf A$ denote the unital algebra generated by the operators $a^+(\xi)$, $a^-(\xi)$,  $\rho(\xi)$ ($\xi\in\mathcal F(X)$). Due to the commutation relations \eqref{xra453}, \eqref{esa4q4z} and the polarization identity, each element of the algebra $\mathbf A$ is a linear combination of the identity operator $\mathbf 1$ and operators of the form $a^+(\varphi)^i\rho(\psi)^ja^-(\xi)^k$ with $\varphi,\psi,\xi\in\mathcal F(X)$ and $i,j,k\in\mathbb N_0$, $i+j+k\ge1$.

  We define the {\em vacuum functional} $\tau_\sigma$ on $\mathbf A$ by setting 
\begin{gather}
\tau_\sigma(\mathbf 1)=1,\notag\\
\tau_\sigma\left(a^+(\varphi)^i\rho(\psi)^ja^-(\xi)^k\right)=0,\quad \varphi,\psi,\xi\in\mathcal F(X),\ i,j,k\in\mathbb N_0,\ i+j+k\ge1,\label{vcyrs5a}
\end{gather}
 and extending it by linearity to the whole $\mathbf A$. 
 
  For each $\xi\in\mathcal F(X)$, we define $R(\xi)\in\mathbf A$ by 
\begin{equation}\label{vcttfts5yw3}
R(\xi):=a^+(\xi)+a^-(\xi)+\rho(\xi)+\langle\sigma,\xi\rangle.\end{equation} 
By the commutation relations \eqref{xra453}, \eqref{esa4q4z}, for any $\xi_1,\xi_2\in\mathcal F(X)$, we have $[R(\xi_1),R(\xi_2)]=0$.

\begin{remark} Consider the representation of the CCR discussed in Remark~\ref{tesw5}. Then each $A\in\mathbf A$ is a linear operator on $\mathcal P(M(X))$. Recall also that $a^+(\xi)$ is the multiplication by $\langle\omega,\xi\rangle$, $a^-(\xi)=\langle\sigma,\xi\partial\rangle$, and $\rho(\xi)=\langle\omega,\xi\partial\rangle$. If $1$ denotes the  monomial on $M(X)$ that is identically equal to 1, then 
$$\big(a^+(\varphi)^i\rho(\psi)^ja^-(\xi)^k 1\big)(\omega)=\begin{cases}
0,&\text{if }\max\{j,k\}\ge1,\\
\langle\omega^{\otimes i},\varphi^{\otimes i}\rangle,&\text{if $i\ge1$ and $j=k=0$}.
\end{cases}
$$
Therefore, by \eqref{vcyrs5a}, the vacuum functional on $\mathbf A$ is given by $\tau_\sigma(A)=(A1)(0)$, i.e., one has to apply the operator $A$ to 1 and then evaluate the obtained polynomial at zero.  Furthermore, it follows from \eqref{vcttfts5yw3} that, in this case, 
$$R(\xi)=\langle\omega+\sigma,\xi(\partial+1)\rangle,\quad\xi\in\mathcal F(X).$$
\end{remark}

\begin{theorem}\label{y6weu43} Let  $\xi_1,\dots,\xi_n\in\mathcal F(X)$, $n\in\mathbb N$ . Under the above assumptions, we then have
$$\tau_\sigma\big(R(\xi_1)\dotsm R(\xi_n)\big)=\mathbb E_\sigma(p). $$
 Here $\mathbb E_\sigma$ is the Poisson functional with intensity measure $\sigma$ and $p(\omega)=\langle\omega,\xi_1\rangle\dotsm \langle\omega,\xi_n\rangle$.

\end{theorem}

Let $\mathbf A'$ denote the commutative unital subalgebra of $\mathbf A$ that is generated by the operators $R(\xi)$ ($\xi\in\mathcal F(X)$). Let $\mathcal P_{\mathrm{alg}}(M(X))$ denote the subset of $\mathcal P(M(X))$  consisting of all polynomials $p$ of the form \eqref{tew53w5}, where each $f^{(k)}$ is a finite sum of functions of the form $\xi_1\odot\dots\odot\xi_k$ with $\xi_1,\dots,\xi_k\in\mathcal F(X)$. Theorem~\ref{y6weu43} states an isomorphism between $(\mathbf A',\tau_\sigma)$ and $(\mathcal P_{\mathrm{alg}}(M(X)),\mathbb E_\sigma)$. Hence, $(\mathbf A',\tau_\sigma)$ can be thought of as the {\it quantum Poisson process with intensity measure $\sigma$}. This is a well-known result in the case where $\sigma$ is a positive non-atomic measure, see e.g.\ \cite{GGPS,HP,Surgailis}.

We first prove the following

\begin{lemma}\label{rqdwes6eqw}
For any $\xi_1,\dots,\xi_n\in\mathcal F(X)$, we have
\begin{align}
&R(\xi_1)\dotsm R(\xi_n)=\sum_{k=1}^n\int_{X^k}(\SO(n,k)(\xi_1\odot\dotsm\odot\xi_n))(x_1,\dots,x_k)\notag\\
&\times
(a^+(x_1)+1)\dotsm(a^+(x_k)+1)(a^-(x_k)+1)\dotsm(a^-(x_1)+1)\,
\sigma^{\otimes k}(dx_1\dotsm dx_k).\label{eraq42q}
\end{align}
\end{lemma}

\begin{proof}
First, we note that the integrals on the right-hand side of formula \eqref{eraq42q} well define  linear operators on the vector space $V$. Indeed, since all creation operators (respectively all annihilation operators) commute, each integral is a finite sum of terms of the form
\begin{align*}
&\int_{X^k}\varphi_1(x_1)\dotsm\varphi_k(x_k)a^+(x_1)\dotsm a^+(x_{i+j})a^-(x_{i+j})\dotsm a^-(x_{i+1})\\
&\qquad\times a^-(x_{i+j+1})\dotsm a^-(x_k)\sigma^{\otimes k}(dx_1\dotsm dx_k)\\
&\quad=a^+(\varphi_1)\dotsm a^+(\varphi_i)\,{:}\rho(\varphi_{i+1})\dotsm\rho(\varphi_{i+j}){:}\,a^-(\varphi_{i+j+1})\dotsm a^-(\varphi_k),
\end{align*}
where $\varphi_1,\dots\varphi_k\in\mathcal F(X)$, $i,j\in\mathbb N_0$, $0\le i+j\le k$.

The formal proof of formula \eqref{eraq42q} is straightforward. Indeed, for $\xi\in\mathcal F(X)$, denote $A^+(\xi):=a^+(\xi)+\langle\sigma,\xi\rangle$, $A^-(\xi):=a^-(\xi)+\langle\sigma,\xi\rangle$. These operators obviously satisfy the CCR. Writing $A^+(\xi)=\int_X \xi(x)A^+(x)\sigma(dx)$ and $A^-(\xi)=\int_X\xi(x)A^-(x)\sigma(dx)$, we get $A^+(x)=a^+(x)+1$, $A^-(x)=a^-(x)+1$. Hence, the corresponding particle density is
$$\int_X\xi(x)A^+(x)A^-(x)\sigma(dx)=\int_X\xi(x)(a^+(x)+1)(a^-(x)+1)\sigma(dx)=R(\xi).$$
Therefore, formula \eqref{eraq42q} follows from Theorem \ref{taw4aq234q}.

To make these arguments rigorous, we proceed as follows. For any $\xi_1,\dots,\xi_n\in\mathcal F(X)$, denote
\begin{align}
{:}R(\xi_1)\dotsm R(\xi_n){:}=&\int_{X^n}\xi_1(x_1)\dotsm \xi_n(x_n) A^+(x_1)
\dotsm A^+(x_n)\notag\\
&\times A^-(x_n)\dotsm A^-(x_1)\sigma^{\otimes n}(dx_1\dotsm dx_n).\label{y6se6qesq}\end{align}
(Recall that the integral on the right-hand side of formula \eqref{y6se6qesq} indeed determines a linear operator on $V$.) Using the commutation relations \eqref{xra453}, \eqref{esa4q4z}, one easily checks that
$${:}R(\xi_1)\dotsm R(\xi_n){:}=R(\xi_1){:}R(\xi_2)\dotsm R(\xi_n){:}-\sum_{i=2}^n {:}R(\xi_2)\dotsm R(\xi_1\xi_i)\dotsm R(\xi_n){:}\,. \qedhere$$
\end{proof}

\begin{proof}[Proof of Theorem~\ref{y6weu43}] Applying the vacuum functional $\tau_\sigma$ to equality \eqref{eraq42q}, we easily obtain
$$\tau_\sigma\big(R(\xi_1)\dotsm R(\xi_n)\big)= \sum_{k=1}^n\langle\sigma^{\otimes k},\SO(n,k)(\xi_1\odot\dots\odot\xi_n)\rangle.$$
Now Theorem \ref{e6ew63} yields the statement. 
\end{proof}

\section{Touchard polynomials}\label{ydrsdese3aqw}

According to Remark \ref{ts5as2}, for each $n,k\in\mathbb N$, $k\le n$, we  define the linear operator $\SO(n,k)^*\in\mathcal L(M^{(k)}(X),M^{(n)}(X))$. For each $n\in\mathbb N$, we define the {\it Touchard (or exponential) polynomial of degree $n$} to be the mapping 
\begin{equation}\label{d5w5w4}
T_n: M(X)\to M^{(n)}(X),\quad T_n(\omega)=\sum_{k=1}^n\SO(n,k)^*\omega^{\otimes k}.\end{equation}
We also set $T_0(\omega):=1$ for all $\omega\in M(X)$.

\begin{remark}
In the case of a single-point space $X$, the polynomials \eqref{d5w5w4} become the classical Touchard polynomials.
\end{remark}

Let $\omega\in M(X)$ and $i\in\mathbb N$. Using a notation from \cite[Section~5]{RW}, we define a {\it diagonal measure} $\omega^{[i]}_{\hat 1}\in M^{(i)}(X)$ by
$$\omega^{[i]}_{\hat 1}(dx_1\dotsm dx_i):=\omega(dx_1)\delta_{x_1}(dx_2)\dotsm\delta_{x_1}(dx_i).$$
By \eqref{cq5y42w}, we have, for each $f^{(i)}\in\mathcal F^{(i)}(X)$,
$$\langle \omega^{[i]}_{\hat 1},f^{(i)}\rangle=\langle\omega,\mathbb D^{(i)}f^{(i)}\rangle.$$
Hence, in view of formulas \eqref{cs5w5a} and \eqref{xsa43wqa4q}, we have, for each $n\in\mathbb N$,
\begin{equation}\label{dtw5ua}
T_n(\omega)=\sum_{k=1}^n\sum_{\lambda=\{\lambda_1,\dots,\lambda_k\}\in \operatorname{UP}(n,k)}\omega^{[\,|\lambda_1|\,]}_{\hat 1}\odot\dots\odot \omega^{[\,|\lambda_k|\,]}_{\hat 1}.\end{equation}

For each $f^{(n)}\in\mathcal F^{(n)}(X)$, 
\begin{equation}\label{ctresw53wq}\langle T_n(\omega),f^{(n)}\rangle=\sum_{k=1}^n\langle\omega^{\otimes k},\SO(n,k)f^{(n)}\rangle\end{equation}
is a polynomial from $\mathcal P(M(X))$, and if $f^{(n)}\in\mathcal {CF}^{(n)}(X)$, then this polynomial belongs to $\mathcal {CP}(M(X))$. 

The following statement is an immediate consequence of  Theorem \ref{e6ew63}.

\begin{corollary}\label{sesawawqa2q}
For each $\omega\in M(X)$,  and $f^{(n)}\in\mathcal F^{(n)}(X)$, $n\in\mathbb N$, we have
\begin{equation*}
\langle T_n(\omega),f^{(n)}\rangle=\mathbb E_\omega(p),
 \end{equation*}
where $p(\omega)=\langle\omega^{\otimes n},f^{(n)}\rangle$. In particular, for each $\xi\in\mathcal F(X)$,
\begin{equation}\label{cydye6edd}
\langle T_n(\omega),\xi^{\otimes n}\rangle=\mathbb E_\omega(m^{(n)}_{\xi}),\end{equation}
where
\begin{equation}\label{vydye}
m_\xi^{(n)}(\omega)=\langle\omega,\xi\rangle^n.
\end{equation}
\end{corollary}

For $\omega\in M(X)$ and $\varphi\in\mathcal F (X)$, we define $\omega\varphi\in M(X)$ by $(\omega\varphi)(dx):=\omega(dx)\varphi(x)$. 
Then, Corollary~\ref{vctrse45w} can be written in the following equivalent form.

\begin{corollary}For any $\xi,\varphi\in\mathcal F(X)$ and $n\in\mathbb N$, we have
\begin{align}
&\langle\omega,\xi\partial\rangle^n\frac1{1-\langle\omega,\varphi\rangle}=
\langle T_n(\omega\varphi),\xi^{\otimes n}\rangle\,\frac{k!}{(1-\langle\omega,\varphi\rangle)^{k+1}},\notag\\
&\langle\omega,\xi\partial\rangle^n e^{\langle\omega,\varphi\rangle}=\langle T_n(\omega\varphi),\xi^{\otimes n}\rangle\, e^{\langle\omega,\varphi\rangle}.\notag
\end{align}

\end{corollary}

The following proposition is a generalization of statement (7) of Theorem~3.29 in \cite{MS}, which is formulated there for the number operator for a pair of operators satisfying CCR.

\begin{proposition}
For each $\xi\in\mathcal F(X)$ and $n\in\mathbb N_0$,
\begin{equation*}
\langle T_{n+1}(\omega),\xi^{\otimes(n+1)}\rangle=
\langle\omega,\xi(\partial+1)\rangle\,\langle T_{n}(\omega),\xi^{\otimes n}\rangle.
\end{equation*}
\end{proposition}

\begin{proof} By \eqref{crtw5w}, \eqref{cresa5}, \eqref{fde5w44} and \eqref{ctresw53wq},
\begin{align*}
&\langle T_{n+1}(\omega),\xi^{\otimes(n+1)}\rangle=\sum_{k=1}^{n+1}\langle\omega^{\otimes k},\SO(n+1,k)\xi^{\otimes (n+1)}\rangle\\
&\quad=\sum_{k=1}^{n+1}\big[
\langle\omega,\xi\rangle\langle\omega^{\otimes(k-1)},\SO(n,k-1)\xi^{\otimes n}\rangle+\langle\omega^{\otimes k},N(\xi)\SO(n,k)\xi^{\otimes n}\rangle\big]\\
&\quad=\sum_{k=1}^{n+1}\big[
\langle\omega,\xi\rangle\langle\omega^{\otimes(k-1)},\SO(n,k-1)\xi^{\otimes n}\rangle+\langle\omega,\xi\partial\rangle\langle\omega^{\otimes k},\SO(n,k)\xi^{\otimes n}\rangle\big]\\
&\quad=\langle\omega,\xi\rangle\sum_{k=2}^{n+1}\langle\omega^{\otimes(k-1)},\SO(n,k-1)\xi^{\otimes n}\rangle+\langle\omega,\xi\partial\rangle\sum_{k=1}^n
\langle\omega^{\otimes k},\SO(n,k)\xi^{\otimes n}\rangle\\
&\quad=\big(\langle\omega,\xi\partial\rangle+\langle\omega,\xi\rangle\big)\langle T_{n}(\omega),\xi^{\otimes n}\rangle.\qedhere
\end{align*}
\end{proof}

The following proposition gives the explicit form of the generating function of the Touchard polynomials.

\begin{proposition}\label{cts65}
 We have
$$\sum_{n=0}^\infty\frac1{n!}\,\langle T_n(\omega),\xi^{\otimes n}\rangle=\exp\big[\langle\omega,e^{\xi}-1\rangle\big],\quad \xi\in\mathcal F(X).$$

\begin{proof}
We will now present two proofs of this result, the first proof giving a connection with  the generating function of the Stirling operators of the second kind, the second proof giving a connection with the Laplace transform of the Poisson functional.

{\it Proof 1}. We have, by \eqref{xtsaer5aw5},
\begin{align*}
&\sum_{n=0}^\infty\frac1{n!}\,\langle T_n(\omega),\xi^{\otimes n}\rangle=
\sum_{n=0}^\infty\frac1{n!}\,\sum_{k=0}^n\langle\omega^{\otimes k},\SO(n,k)\xi^{\otimes n}\rangle\\
&\quad=\sum_{k=0}^\infty\Big\langle\omega^{\otimes k},\,\sum_{n=k}^\infty\frac1{n!}\,\SO(n,k)\xi^{\otimes n}\Big\rangle=\sum_{k=0}^\infty\frac1{k!}\,\langle\omega^{\otimes k},(e^\xi-1)^{\otimes k}\rangle\\
&\quad=\sum_{k=0}^\infty\frac1{k!}\,\langle\omega,e^\xi-1\rangle^k=\exp\big[\langle\omega,e^{\xi}-1\rangle\big].
\end{align*}

{\it Proof 2}. Using formulas  \eqref{cydye6edd}, \eqref{vydye}, \eqref{vuffr7}, and Lemma \ref{xseaq43},
  we have
$$
\sum_{n=0}^\infty\frac1{n!}\,\langle T_n(\omega),\xi^{\otimes n}\rangle=
\sum_{n=0}^\infty \frac1{n!}\,\mathbb E_\omega(m^{(n)}_\xi)=\exp\big[\langle\omega,e^{\xi}-1\rangle\big].\qquad \qedhere
$$
\end{proof}

\end{proposition}

\begin{corollary}
The Touchard polynomials satisfy the binomial property:
$$T_n(\omega+\sigma)=\sum_{k=0}^n\binom nk T_k(\omega)\odot T_{n-k}(\sigma),\quad\omega,\sigma\in M(X),\ n\in\mathbb N.$$
\end{corollary}

\begin{proof}
The statement follows immediately from Proposition \ref{cts65} and the proof of the implication $\mathrm{(BT4)}\Rightarrow \mathrm{(BT1)}$ in 
\cite[Theorem 4.1]{FKLO}. 
\end{proof}

\begin{proposition}\label{rseaopti}
For each $n\in\mathbb N_0$ and $\omega\in M(X)$,
$$T_{n+1}(\omega)=\sum_{k=0}^n\binom nk T_k(\omega)\odot\omega^{[n+1-k]}_{\hat 1}. $$
\end{proposition}

\begin{proof}
For each $\xi\in\mathcal F(X)$ with support in $\Lambda\in\mathcal B_0(X)$ and $\omega\in M(X)$, we have, by Proposition~\ref{g6ew4} and formulas  \eqref{cydye6edd}, \eqref{vydye},
and \eqref{tyfdqwd6rqwx},
\begin{align*}
\langle T_{n+1}(\omega),\xi^{\otimes(n+1)}\rangle&=\mathbb E_\omega(m^{(n+1)}_\xi)\\
&=\int_{\ddot\Gamma_0(\Lambda)}\mathbb P_\omega^\Lambda(d\eta)\int_\Lambda\eta(dx)\xi(x)\langle\eta,\xi\rangle^n\\
&=\int_{\ddot\Gamma_0(\Lambda)}\mathbb P_\omega^\Lambda(d\eta)\int_\Lambda\omega(dx)\xi(x)\langle\eta+\delta_x,\xi\rangle^n\\
&=\sum_{k=0}^n\binom nk\int_{\ddot\Gamma_0(\Lambda)}\mathbb P_\omega^\Lambda(d\eta)\langle\eta,\xi\rangle^k\int_\Lambda\omega(dx)\xi^{n+1-k}(x)\\
&=\Big\langle\sum_{k=0}^n\binom nk T_k(\omega)\odot\omega^{[n+1-k]}_{\hat 1},\xi^{\otimes (n+1)}\Big\rangle.\qedhere
\end{align*}
\end{proof}

Let $M_1(X)\subset M(X)$ denote the set of all probability measures on $X$. 
For $\nu\in M_1(X)$, we define  $B_n(\nu):=T_n(\nu)\in M^{(n)}(X)$ and call $B_n(\nu)$ the {\it $n$th Bell measure corresponding to $\nu$}. By \eqref{dtw5ua}, $B_n(\nu)$ is a finite positive measure on $X^n$ such that the  $B_n(\nu)$ measure of the whole $X^n$ is equal to $B_n$, the classical Bell number.

\begin{proposition}\label{qwdwswxe}
Let $n\in\mathbb N$ and $f^{(n)}:X^n\to\mathbb F$ be measurable and bounded. Let $\nu\in M_1(X)$ and let $(Z_n)_{n=1}^\infty$ be a sequence of independent, identically distributed $\mathbb F$-valued random variables whose distribution is equal to $\nu$. Let $\mathbb E(\cdot)$ denote the corresponding expectation. Then
\begin{equation}\label{re64u4e}
\langle B_n(\nu),f^{(n)}\rangle=e^{-1}\,\mathbb E\bigg(1+\sum_{k=1}^\infty\frac1{k!}\sum_{i_1=1}^k\dotsm\sum_{i_n=1}^k f^{(n)}(Z_{i_1},\dots,Z_{i_n})\bigg).\end{equation}
\end{proposition}

\begin{proof}Without loss of generality we may assume that the function $f^{(n)}$ is symmetric. If additionally, $f^{(n)}$ has a compact support, i.e., $f^{(n)}\in\mathcal F^{(n)}(X)$, then formula \eqref{re64u4e} easily follows from Corollary~\ref{sesawawqa2q} and Lemma~\ref{gdxtssd}.  For a general function $f^{(n)}$, formula~\eqref{re64u4e} can then be shown by approximating $f^{(n)}$ with functions $f^{(n)}\chi_{k}^{(n)}$, where $\chi_{k}^{(n)}$ is the indicator function of the set $\Lambda_k^n$, and 
$\{\Lambda_k\}_{k=1}^\infty$ is an increasing sequence of compact subsets of $X$ such that  $\Lambda_k\uparrow\ X$. \end{proof}

\begin{remark}
By setting $f^{(n)}\equiv 1$ in formula \eqref{re64u4e}, we obtain Dobi\'nski's formula $B_n=e^{-1}(1+\sum_{k=1}^\infty\frac{k^n}{k!})$.
\end{remark}

For $\nu\in M_1(X)$, we define $D_n(\nu):=T_n(-\nu)\in M^{(n)}(X)$ and call $D_n(\nu)$ the {\it $n$th Rao--Uppuluri--Carpenter measure corresponding to $\nu$}. The measures $D_n(\nu)$ generalize the corresponding numbers $D_n$ in classical combinatorics, see e.g. \cite[Section 9.4]{QG}.   The $D_n(\nu)$ is a signed measure on $X^n$ of finite total variation, and the $D_n(\nu)$ measure of the whole $X^n$ is equal to $D_n$. 
Proposition \ref{rseaopti} implies
$$D_{n+1}(\nu)=-\sum_{k=0}^n\binom nk D_k(\nu)\odot\nu^{[n+1-k]}_{\hat 1}. $$
Under the assumptions of Proposition \ref{qwdwswxe}, we easily obtain
$$\langle D_n(\nu),f^{(n)}\rangle=e\,\mathbb E\bigg(1+\sum_{k=1}^\infty\frac{(-1)^k}{k!}\sum_{i_1=1}^k\dotsm\sum_{i_n=1}^k f^{(n)}(X_{i_1},\dots,X_{i_n})\bigg).$$

\section{An open problem}\label{xzrwaq4yq4q}

Among several open problems arising from the present paper, let us mention only a problem related to Theorem~\ref{taw4aq234q}.

The idea of Katriel's theorem \cite{Katriel} was used in order to define and study several deformations of Stirling numbers, see e.g.\ \cite{KK92,MS}.  
One may similarly use the idea of Theorem~\ref{taw4aq234q} in order to define and study generalizations of  Stirling operators in the case of various deformed commutation relations, see e.g.\ \cite{BEH,bozejkolytvynovwysoczanskiCMP2012,BLW-Q,BS}. Indeed, assume that we are given a family of creation operators $a^+(\xi)$ and annihilation operators $a^-(\xi)$ ($\xi\in\mathcal F(X)$) that satisfy certain deformed commutation relations. Then, similarly to the CCR case, we can define operators $\rho(\xi)=\int_X\xi(x)a^+(x)a^-(x)\sigma(dx)$ and the Wick product ${:}\,\rho(x_1)\dotsm \rho(x_n){:}$ by formula \eqref{cd55uw}. Now, the question is: Can we define deformed Stirling operators, $\SO(n,k)$, in such a way that formula \eqref{ds5w5u4}  holds in this setting? In the following two cases, one can easily answer this question. 

First, consider the case of the anyon commutation relations \cite{bozejkolytvynovwysoczanskiCMP2012},
\begin{gather*}
a^+(x)a^+(y)-Q(y,x)a^+(x)a^+(y)=0,\quad a^-(x)a^-(y)-Q(y,x)a^-(x)a^-(y)=0,\\
a^-(x)a^+(y)-Q(x,y)a^+(y)a^-(x)=\delta(x,y).
\end{gather*} 
Here $Q(x,y)$ is a complex-valued function on $X^2$ such that $|Q(x,y)|=1$ and $Q(y,x)=\overline{Q(x,y)}$. An easy calculation shows that   Lemma~\ref{erws5w3}  holds true in this case, hence the operators $\SO(n,k)$ remain the same as in the case of the CCR. 

Second, consider the case of the free commutation relation, 
$$a^-(x)a^+(y)=\delta(x,y). $$
 As easily seen, 
   $$\rho(\xi_1)\dotsm\rho(\xi_n)=\rho(\xi_1\dotsm\xi_n)$$
   for any $\xi_1,\dots,\xi_n\in\mathcal F(X)$. 
 Hence, in this case, $\SO(n,1)$ is the operator $\mathbb D^{(n)}$ given by \eqref{cq5y42w}, and $\SO(n,k)=0$ if $k\ne1$.

Now one can naturally ask which generalizations of the Stirling operators appear if the operator-valued distributions $(a^+(x),a^-(x))_{x\in X}$ satisfy the $q$-commutation relation
for $q\in(-1,1)$ with  $q\ne0$:
\begin{equation}\label{vcyei76vhvv}
a^-(x)a^+(y)-qa^+(y)a^-(x)=\delta(x,y),\end{equation}
see \cite{BS}. If the set $X$ has a single point, the commutation relation \eqref{vcyei76vhvv} becomes
$$a^-a^+-qa^+a^-=1,$$
compare with \eqref{xzzaaaa}. In this one-mode case, the corresponding Stirling  numbers were studied in \cite{KK92}, see also  \cite[Section 7.2]{MS} and the references therein. However, in the general case of  \eqref{vcyei76vhvv}, an explicit representation of the product $\rho(\xi_1)\dotsm\rho(\xi_n)$ through Wick ordered operators is an open, interesting problem.

One may also extend this problem to the case of the commutation relation discussed in \cite{BLW-Q}, in which case the parameter $q$ is replaced by a complex-valued function $Q(x,y)$ on $X^2$ such that $|Q(x,y)|\le 1$ and $Q(y,x)=\overline{Q(x,y)}$. Another important generalization of \eqref{vcyei76vhvv} was considered in \cite{BEH} and is related to Coxeter groups of type B.

\subsection*{Acknowledgments}

MJO was supported by the Portuguese 
national funds through FCT--Funda{\c c}\~ao para a Ci\^encia e a Tecnologia, 
within the project UIDB/04561/2020.  EL and MJO are grateful to the London Mathematical Society for partially supporting the visit of MJO to Swansea University.

\appendix
\renewcommand{\thesection}{A}
\section*{Appendix: Poisson functional}
\setcounter{theorem}{0}
\setcounter{equation}{0}

Let $\Lambda\in\mathcal B_0(X)$. Denote $\ddot\Gamma_0(\Lambda):=\{\eta\in\ddot\Gamma_0(X)\mid\operatorname{Supp}(\eta)\subset\Lambda\}$,  where $\operatorname{Supp}(\eta)$ denotes the support of $\eta$. Let $\mathcal B(\ddot\Gamma_0(\Lambda))$ denote the minimal $\sigma$-algebra on $\ddot\Gamma_0(\Lambda)$ such that, for each  $n\in\mathbb N$, the mapping $\Lambda^n\ni(x_1,\dots,x_n)\mapsto[x_1,\dots,x_n]\in \ddot\Gamma_0(\Lambda)$
is measurable. 

Fix an arbitrary $\omega\in M(X)$. For each  $\Lambda\in\mathcal B_0(X)$, we define an $\mathbb F$-valued  measure $\mathbb P_\omega^\Lambda$ on $(\ddot\Gamma_0(\Lambda),\mathcal B(\ddot\Gamma_0(\Lambda)))$ that satisfies, for each measurable bounded function $F:\ddot\Gamma_0(\Lambda)\to\mathbb F$,
\begin{equation}\label{tyfdqwd6rqwx}
\int_{\ddot\Gamma_0(\Lambda)}F(\eta) \,\mathbb P_\omega^\Lambda(d\eta)=e^{-\omega(\Lambda)}\sum_{n=0}^\infty \frac1{n!}\int_{\Lambda^n}
F([x_1,\dots,x_n])\,\omega^{\otimes n}(dx_1\dotsm dx_n).\end{equation} 
Obviously, the above formula holds true for any measurable function  
$F:\ddot\Gamma_0(\Lambda)\to\mathbb F$ that is integrable with respect to the total variation of $\mathbb P_\omega^\Lambda$. 

\begin{remark}\label{gvyufvvg}
If $\omega\in M(X)$ has a finite total variation, we may similarly define an $\mathbb F$-valued measure $\mathbb P_\omega:=\mathbb P_\omega^X$ on $(\ddot\Gamma_0(X),\mathcal B(\ddot\Gamma_0(X)))$.
\end{remark}

\begin{lemma}\label{rde2ed}
 {\rm (i)} Let $\Lambda_1,\Lambda_2\in\mathcal B_0(X)$ be such that $\Lambda_1\cap\Lambda_2=\varnothing$. Denote $\Lambda:=\Lambda_1\cup\Lambda_2$. Let $F_1,F_2:\ddot\Gamma_0(\Lambda)\to\mathbb F$ be bounded measurable functions such that, for all $\eta\in\ddot\Gamma_0(\Lambda)$, $F_i(\eta)=F_i(\eta_{\Lambda_i})$, $i=1,2$. Here $(\eta_{\Lambda_i})(dx):=\eta(dx)\chi_{\Lambda_i}(x)$. Then
$$\int_{\ddot\Gamma_0(\Lambda)}F_1(\eta)F_2(\eta)\,\mathbb P_\omega^{\Lambda}(d\eta)=\int_{\ddot\Gamma_0(\Lambda_1)}F_1(\eta_1)\,\mathbb P_\omega^{\Lambda_1}(d\eta_1)\int_{\ddot\Gamma_0(\Lambda_2)}F_2(\eta_2)\,\mathbb P_\omega^{\Lambda_2}(d\eta_2).$$

{\rm (ii)} Let $\Lambda,\Lambda_1\in\mathcal B_0(X)$ be such that $\Lambda_1\subset\Lambda$.  Let $F:\ddot\Gamma_0(\Lambda)\to\mathbb F$ be a bounded measurable function such that, for each $\eta\in \ddot\Gamma_0(\Lambda)$, $F(\eta)=F(\eta_{\Lambda_1})$.  Then 
$$\int_{\ddot\Gamma_0(\Lambda)}F(\eta) \,\mathbb P_\omega^{\Lambda}(d\eta)=\int_{\ddot\Gamma_0(\Lambda_1)}F(\eta) \,\mathbb P_\omega^{\Lambda_1}(d\eta).$$
\end{lemma} 

\begin{proof}
(i) We have
\begin{align*}
&\int_{\ddot\Gamma_0(\Lambda)}F_1(\eta)F_2(\eta)\,\mathbb P_\omega^{\Lambda}(d\eta)\\
&\quad=e^{-\omega(\Lambda)}\sum_{n=0}^\infty\frac1{n!}\int_{\Lambda^n}F_1([x_1,\dots,x_n]_{\Lambda_1})F_2([x_1,\dots,x_n]_{\Lambda_2})\,\omega^{\otimes n}(dx_1\dotsm dx_n)\\
&\quad=e^{-\omega(\Lambda_1)-\omega(\Lambda_2)}\sum_{n=0}^\infty\frac1{n!}
\sum_{k=0}^n\binom nk\int_{\Lambda_1^k}F_1([x_1,\dots,x_k])\,\omega^{\otimes k}(dx_1\dotsm dx_k)\\
&\qquad\times \int_{\Lambda_2^{n-k}}F_2([y_1,\dots,y_{n-k}])\,\omega^{\otimes (n-k)}(dy_1\dotsm dy_{n-k})\\
&\quad=\int_{\ddot\Gamma_0(\Lambda_1)}F_1(\eta_1)\,\mathbb P_\omega^{\Lambda_1}(d\eta_1)\int_{\ddot\Gamma_0(\Lambda_2)}F_2(\eta_2)\,\mathbb P_\omega^{\Lambda_2}(d\eta_2).
\end{align*}

(ii) Define $\Lambda_2:=\Lambda\setminus\Lambda_1$, $F_1(\eta):=F(\eta)$, $F_2(\eta)=1$ and apply (i). 
\end{proof}

Let $p\in\mathcal P(M(X))$ be of the form $p(\omega)=\sum_{k=0}^n\langle\omega^{\otimes k},f^{(k)}\rangle$ and choose an arbitrary $\Lambda\in \mathcal B_0(X)$ such  that, for each $k=1,\dots,n$, the function $f^{(k)}$ vanishes outside $\Lambda^k$. Then, for each $\eta\in\ddot\Gamma_0(X)$, $p(\eta)=p(\eta_\Lambda)$. Therefore, in view of  Lemma~\ref{rde2ed}, we can define $\mathbb E_\omega(p):=\int_{\ddot\Gamma_0(\Lambda)}p(\eta)\,\mathbb P^\Lambda_\omega(d\eta)$, and  this definition does not depend on the choice of a set $\Lambda$.  $\mathbb E_\omega$ is the Poisson functional with intensity measure $\omega$, compare \eqref{ufyde} and \eqref{tyfdqwd6rqwx}. 

\begin{remark}

\end{remark}

\begin{lemma}\label{gdxtssd}
Assume additionally that $\omega\in M(X)$ has a finite total variation.
Then $\mathbb E_\omega(p)=\int_{\ddot\Gamma_0(X)}p(\eta)\,\mathbb P_\omega(d\eta)$
for all $p\in\mathcal P(M(X))$.
\end{lemma}

\begin{proof}
Similar to the proof of Lemma~\ref{rde2ed} (ii).
\end{proof}

\begin{remark}\label{ydqry} In the case where $\omega\in M(X)$ is an infinite  positive measure (not necessarily non-atomic), one can prove, by using Kolmogorov's theorem, that there exists a unique probability measure $\mathbb P_\omega$ on $\ddot\Gamma(X)$, the space of multiple configurations in $X$, such that $\mathbb E_\omega(p)=\int_{\ddot\Gamma(X)}p(\gamma)\,\mathbb P_\omega(d\gamma)$ for all $p\in\mathcal P(M(X))$.  If $\omega\in M(X)$ is neither  positive nor does it have  a finite total variation, then the Poisson functional $\mathbb E_\omega$ is not given by an $\mathbb F$-valued measure  $\mu$ on $\ddot\Gamma(X)$.
\end{remark}

For $\xi\in\mathcal F(X)$, we denote $E_\xi(\omega)=e^{\langle \omega,\xi\rangle}$ for $\omega\in M(X)$.
For a fixed $\omega\in M(X)$, we  define
\begin{equation}\label{vuffr7}
\mathbb E_\omega(E_\xi):=\sum_{n=0}^\infty \frac1{n!}\,\mathbb E_\omega (m^{(n)}_\xi),\end{equation}
where $m^{(n)}_\xi$ is defined by \eqref{vydye}. 
 As easily seen, the sum on the right hand side of formula \eqref{vuffr7} converges absolutely. The map $\mathcal F(X)\ni \xi\mapsto \mathbb E_\omega(E_\xi)\in\mathbb F$ can be thought of as the Laplace transform of the Poisson functional $\mathbb E_\omega$.

\begin{lemma}\label{xseaq43} For each $\xi\in\mathcal F(X)$,
$$\mathbb E_\omega(E_\xi) =\exp\big[\langle\omega,e^\xi-1\rangle\big].$$
\end{lemma}

\begin{proof}
It follows from \eqref{vuffr7} that, if the function $f\in \mathcal F(X)$ vanishes outside $\Lambda\in\mathcal B_0(X)$, then 
\begin{equation}\label{vcyds5uwq3}
\mathbb E_\omega(E_\xi)=\int_{\ddot\Gamma_0(\Lambda)}e^{\langle \eta,\xi\rangle}\,\mathbb P^\Lambda_\omega(d\eta).  \end{equation}
From  \eqref{tyfdqwd6rqwx}  and \eqref{vcyds5uwq3} the statement immediately follows.
\end{proof}

The following proposition is an extension of the Mecke identity for a Poisson point process \cite{Mec67}.

\begin{proposition}\label{g6ew4} 
Let $\Lambda\in\mathcal B_0(X)$ and let $F:\ddot\Gamma_0(\Lambda)\times\Lambda$ be measurable and bounded. Then
$$\int_{\ddot\Gamma_0(\Lambda)}\mathbb P_\omega^{\Lambda}(d\eta)\int_\Lambda\eta(dx)F(\eta,x)=\int_{\ddot\Gamma_0(\Lambda)}\mathbb P_\omega^{\Lambda}(d\eta)\int_\Lambda\omega(dx)F(\eta+\delta_x,x).$$
\end{proposition}

\begin{proof}We have
\begin{align*}
&\int_{\ddot\Gamma_0(\Lambda)}\mathbb P_\omega^{\Lambda}(d\eta)\int_\Lambda\eta(dx)F(\eta,x)=e^{-\omega(\Lambda)}\sum_{n=1}^\infty\frac1{n!}
\int_{\Lambda^n}\omega^{\otimes n}(dx_1\dotsm dx_n)\sum_{i=1}^n F([x_1,\dots,x_n],x_i)\\
&\quad= e^{-\omega(\Lambda)}\sum_{n=1}^\infty\frac{n}{n!}\int_\Lambda\omega(dx)\int_{\Lambda^{n-1}}\omega^{\otimes(n-1)}(dx_1\dotsm dx_{n-1})F([x_1,\dots,x_{n-1},x],x)\\
&\quad=\int_{\ddot\Gamma_0(\Lambda)}\mathbb P_\omega^{\Lambda}(d\eta)\int_\Lambda\omega(dx)F(\eta+\delta_x,x).\qedhere
\end{align*}
\end{proof}

\end{document}